\documentclass[11pt]{article}

\usepackage{amssymb,amsthm,amsmath,mathrsfs,fullpage}
\usepackage{tikz}
\usetikzlibrary{matrix}

\newcommand{\rs}{\Sigma^-}
\newcommand{\pihat}{\hat{\pi}}
\newcommand{\tent}{\Lambda}
\newcommand{\st}{\mbox{st}}

\DeclareMathOperator{\Pat}{Pat}
\newcommand{\PPat}{\Pi}
\newcommand{\PPPat}{\overline{\Pi}}
\DeclareMathOperator{\Av}{Av}
\DeclareMathOperator{\des}{des}
\DeclareMathOperator{\Des}{Des}
\DeclareMathOperator{\asc}{asc}
\DeclareMathOperator{\Asc}{Asc}
\newcommand{\Allow}{\mathcal A}
\renewcommand\S{\mathcal S}
\renewcommand\SS{\overline{\S}}
\newcommand\C{\mathcal C}
\newcommand\hp{\hat{\pi}}
\newcommand\W{\mathcal W}
\newcommand\WW{\overline{\mathcal W}}
\newcommand\PP{\mathcal P}
\newcommand\PPP{\overline{\mathcal P}}
\renewcommand\O{\mathcal O}
\newcommand\lex{<_\mathrm{lex}}

\newtheorem{theorem}{Theorem}[section]
\newtheorem{lemma}[theorem]{Lemma}
\newtheorem{proposition}[theorem]{Proposition}
\newtheorem{corollary}[theorem]{Corollary}

\title{Cyclic permutations realized by signed shifts}
\author{Kassie Archer and Sergi Elizalde~\thanks{Department of Mathematics,
Dartmouth College, Hanover, NH 03755.}
\thanks{Partially supported by NSF grant DMS-1001046.}}
\date{}

\begin{document}

\maketitle

\begin{abstract}
The periodic (ordinal) patterns of a map are the permutations realized by the relative order of the points in its periodic orbits.
We give a combinatorial characterization of the periodic patterns of an arbitrary signed shift,
in terms of the structure of the descent set of a certain cyclic permutation associated to the pattern.
Signed shifts are an important family of one-dimensional dynamical systems that includes shift maps and the tent map as particular cases.
Defined as a function on the set of infinite words on a finite alphabet, a signed shift deletes the first letter and, depending on its value,
possibly applies the complementation operation on the remaining word.
 For shift maps, reverse shift maps, and the tent map, we give exact formulas for their number of periodic patterns.
As a byproduct of our work, we recover results of Gessel--Reutenauer and Weiss--Rogers and obtain new enumeration formulas for pattern-avoiding cycles.
\end{abstract}


%\noindent {\bf MSC 2010:} 05A15 (primary); 37M10, 05A15, 94A55 (secondary).

\section{Introduction}

\subsection{Background and motivation}

Permutations realized by the orbits of a map on a one-dimensional interval have received a significant amount of attention in the last few years~\cite{Amigobook}.
These are the permutations given by the relative order of the elements of the sequence obtained by successively iterating the map, starting from any point in the interval.
One the one hand, understanding these permutations provides a powerful tool to distinguish random from deterministic time series,
based on the remarkable fact~\cite{BKP} that every piecewise monotone map has forbidden patterns, i.e.,
permutations that are not realized by any orbit. Permutation-based tests for this purpose have been developed in~\cite{AZS,AZS2}.
On the other hand, the set of permutations realized by a map (also called allowed patterns) is closed under consecutive pattern containment. This gives rise to enumerative questions about
pattern-avoiding permutations whose answer provides information about the associated dynamical systems. For example, determining the asymptotic growth of the number of
allowed patterns of a map reveals its so-called topological entropy, an important measure of the complexity of the system.

Among the dynamical systems most commonly studied from the perspective of forbidden patterns are shifts, and more generally signed shifts~\cite{Amigosigned}. Signed shifts form a large family of maps that includes
the tent map, which is equivalent to the logistic map in terms of forbidden patterns.
As we will see, signed shifts have a simple discrete structure which makes them amenable to a combinatorial approach, yet they include many important chaotic dynamical systems.

Permutations realized by shifts were first considered in~\cite{AEK}, and later characterized and enumerated in~\cite{Elishifts}. More recently, permutations realized by the more general $\beta$-shifts have been studied in~\cite{Elibeta}.
For the logistic map, some properties of their set of forbidden patterns were given in~\cite{Eliu}.

If instead of considering an arbitrary initial point in the domain of the map we restrict our attention to periodic points, the permutations realized by the relative order of the entries in the corresponding orbits (up until the first repetition) are called {\em periodic patterns}.
In the case of continuous maps, Sharkovskii's theorem~\cite{Sarko} gives a beautiful characterization of the possible periods of these orbits.
More refined results that consider which periodic patterns are forced by others are known for continuous maps~\cite{BCMM,Blo,Bobok,Jun}. In an equivalent form, periodic orbits of the tent map were studied in~\cite{WR}
in connection to bifurcations of stable periodic orbits in a certain family of quadratic maps.
However, little is known when the maps are not continuous, as is the case for shifts and for most signed shifts.

The subject of study of this paper are periodic patterns of signed shifts. Our main result is a characterization of the periodic patterns of an (almost) arbitrary signed shift, given in Theorem~\ref{thm: description of periodic patterns}.
For some particular cases of signed shifts we obtain exact enumeration formulas: the number of periodic patterns of the tent map is given in Theorem~\ref{thm:enumtent}, recovering a formula of Weiss and Rogers~\cite{WR}, and the number of periodic patterns of the (unsigned) shift map is given in Theorem~\ref{thm: enum shift}. For the reverse shift, which is not covered in our main theorem, formulas for the number of periodic patterns, which depend on the residue class of $n\bmod 4$, are given in Sections~\ref{sec:reverse} and~\ref{sec:reverse2odd}.

An interesting consequence of our study of periodic patterns is that we obtain new results (and some old ones)
regarding the enumeration of cyclic permutations that avoid certain patterns. These are described in Section~\ref{sec:patternsincycles}.

\subsection{Periodic patterns}

Given a linearly ordered set $X$ and a map $f:X\to X$, consider the sequence $\{f^i(x)\}_{i\ge0}$
obtained by iterating the function starting at a point $x\in X$.
If there are no repetitions among the first $n$ elements of this sequence, called the {\em orbit} of $x$,
then we define the {\em pattern} of length $n$ of $f$ at $x$ to be
$$\Pat(x, f, n) = \st(x, f(x), f^2(x), \dots, f^{n-1}(x)),$$
where $\st$ is the {\em reduction} operation that outputs the permutation of $[n]= \{1,2,\dots, n\}$ whose entries are in the same relative order as $n$ entries in the input.
For example, $\st(3.3,3.7,9,6,0.2)=23541$. If $f^i(x)=f^j(x)$ for some $0\le i<j<n$, then $\Pat(x, f, n)$ is not defined. The set of {\em allowed patterns} of $f$ is
$$\Allow(f)=\{\Pat(x,f,n):n\ge0,x\in X\}.$$

We say that $x \in X$ is an {\em $n$-periodic point} of $f$ if $f^n(x) = x$ but $f^i(x)\neq x$ for $1\le i<n$, and in this case the set $\{f^i(x):0\le i<n\}$ is called an $n$-periodic orbit.
If $x$ is an $n$-periodic point, the permutation $\Pat(x, f, n)$ is denoted $\PPat(x,f)$, and is called the {\em periodic pattern} of $f$ at $x$.
Let $$\PP_n(f)=\{\PPat(x,f):x\in X \mbox{ is an $n$-periodic point of $f$}\},$$ and let
$\PP(f)=\bigcup_{n\ge0}\PP_n(f)$ be the set of {\em periodic patterns} of $f$.
For a permutation $\pi=\pi_1\pi_2\dots\pi_n\in\S_n$, let $[\pi]=\{\pi_i\pi_{i+1}\dots\pi_n\pi_1\dots\pi_{i-1}:1\le i\le n\}$ denote the set of cyclic rotations of $\pi$, which we call the {\em equivalence class} of $\pi$.
It is clear that if $\pi\in\PP(f)$, then $[\pi]\subset\PP(f)$. Indeed, if $\pi$ is the periodic pattern at a point $x$, then the other permutations in $[\pi]$ are realized at the other points in the periodic orbit of $x$.
Thus it is convenient to define the set $\PPP_n(f)=\{[\pi]:\pi\in\PP_n(f)\}$, consisting of the equivalence classes of periodic patterns of $f$ of length $n$. We will be interested in finding the cardinality of this for certain maps $f$, which we denote $p_n(f)=|\PPP_n(f)|=|\PP_n(f)|/n$.

Given linearly ordered sets $X$ and $Y$, two maps $f:X\to X$ and $g:Y\to Y$ are said to be {\em order-isomorphic} if there is an order-preserving bijection $\phi:X\to Y$ such that
$\phi\circ f=g\circ\phi$. In this case, $\Pat(x,f,n)=\Pat(\phi(x),g,n)$ for every $x\in X$ and $n\ge1$. In particular, $\Allow(f)=\Allow(g)$ and $\PP(f)=\PP(g)$.

\subsection{Signed shifts}

Let $k\ge2$ be fixed, and let $\W_k$ be the set of infinite words $s=s_1s_2\dots$ over the alphabet $\{0,1,\dots,k-1\}$. Let $\lex$ denote the lexicographic order on these words.
We use the notation $s_{[i,\infty)}=s_is_{i+1}\dots$, and $\bar{s_i}=k-1-s_i$. If $q$ is a finite word, $q^m$ denotes concatenation of $q$ with itself $m$ times, and $q^\infty$ is an infinite periodic word.

Fix $\sigma=\sigma_0\sigma_1\dots\sigma_{k-1} \in \{+,-\}^k$. Let $T^+_\sigma = \{t :  \sigma_t = +\}$ and $T^-_\sigma=\{ t  :  \sigma_t = -\}$, and note that these sets form a partition of $\{0,1,\dots,k-1\}$.
We give two definitions of the signed shift with signature $\sigma$, and show that they are order-isomorphic to each other.

The first definition, which we denote by $\Sigma'_\sigma$, is the map $\Sigma'_\sigma:(\W_k,\lex)\to(\W_k,\lex)$ defined by
$$\Sigma'_\sigma(s_1s_2s_3s_4\dots)=\begin{cases} s_2s_3s_4\dots & \mbox{if }s_1\in T^+_\sigma, \\
\bar{s_2}\bar{s_3}\bar{s_4}\dots & \mbox{if }s_1\in T^-_\sigma. \end{cases}$$
The order-preserving transformation
$$\begin{array}{cccc}  \phi_k:&(\W_k,\lex) &\to&([0,1],<) \\
&s_1s_2s_3s_4 \dots &\mapsto& \sum_{i\ge0}  s_i k^{-i-1}
\end{array}$$
can be used to show (see~\cite{Amigosigned}) that $\Sigma'_\sigma$ is order-isomorphic to the piecewise linear function $M_\sigma:[0,1]\to[0,1]$ defined
for $x \in [\frac{t}{k}, \frac{t+1}{k})$, for each $0\le t\le k-1$, as
$$M_\sigma(x)=\begin{cases} kx -t & \mbox{if } t \in T^+_\sigma, \\ t+1-kx & \mbox{if }t \in T^-_\sigma.\end{cases}$$
As a consequence, the allowed patterns and the periodic patterns of $\Sigma'_\sigma$ are the same as those of $M_\sigma$, respectively.
A few examples of the function $M_\sigma$ are pictured in Figure~\ref{fig:signedshifts}.

\begin{figure}[ht]
\centering
   \includegraphics[width=.23\linewidth]{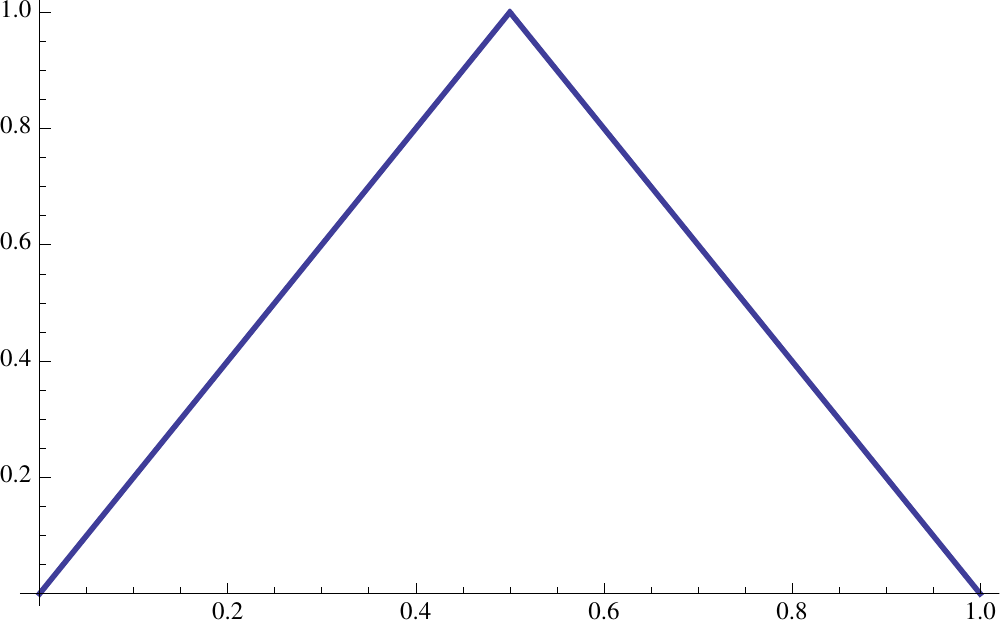}\hspace{.01\linewidth}
    \includegraphics[width=.23\linewidth]{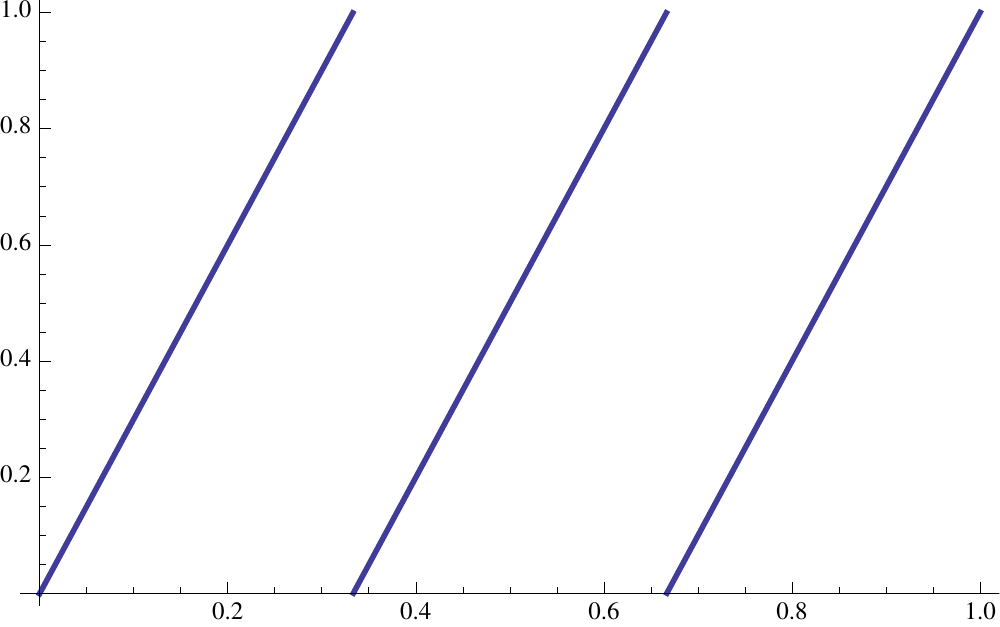}\hspace{.01\linewidth}
    \includegraphics[width=.23 \linewidth]{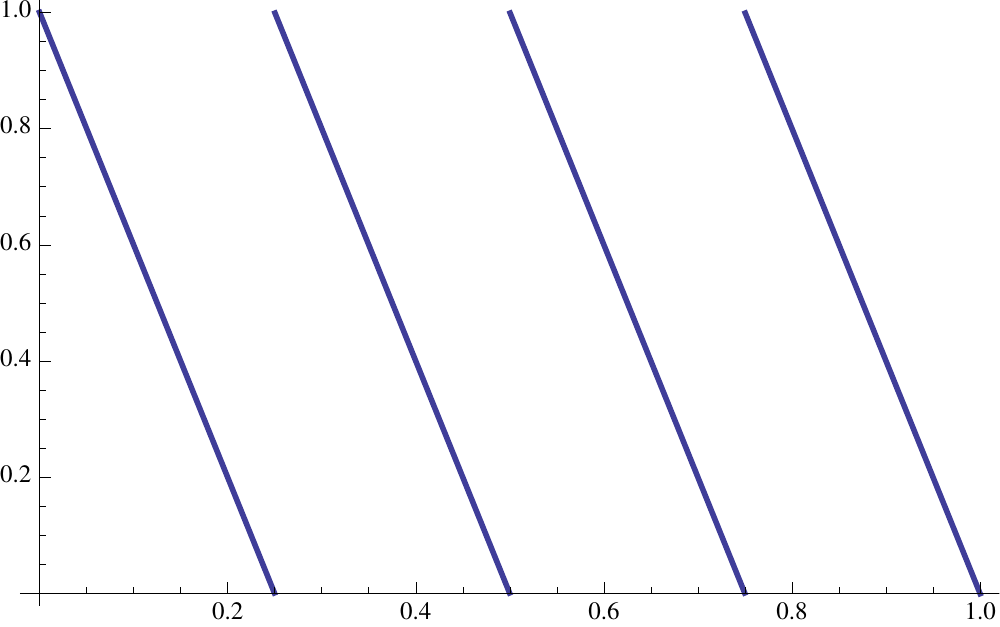}\hspace{.01\linewidth}
    \includegraphics[width=.23\linewidth]{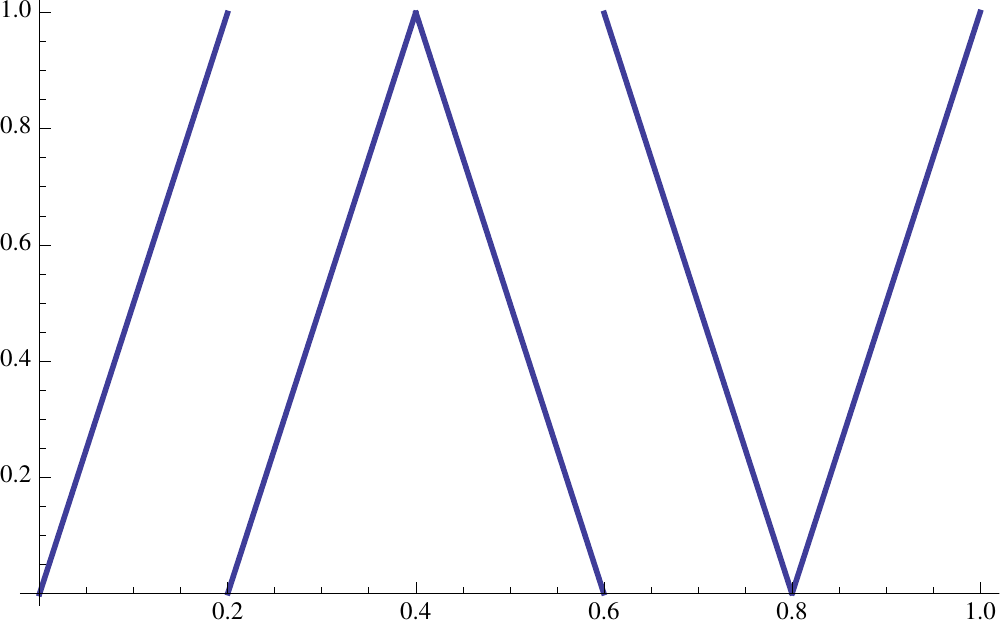}
\caption{The graphs of $M_\sigma$ for $\sigma =+-$, $\sigma =+++$, $\sigma =----$ and $\sigma = + + - - +$, respectively.}
\label{fig:signedshifts}
\end{figure}

We next give another definition of the signed shift that will be more convenient when studying its periodic patterns.
Let $\prec_\sigma$ be the linear order on $\W_k$ defined by $s= s_1s_2s_3\dots\prec_\sigma t_1t_2t_3\dots = t$ if one of the following holds:
\begin{enumerate}
\item $s_1 < t_1$,
\item $s_1 = t_1 \in T^+_\sigma$ and $s_2s_3\dots\prec_\sigma t_2t_3\dots$, or
\item $s_1 = t_1 \in T^-_\sigma$ and $t_2t_3\dots\prec_\sigma s_2s_3\dots$.
\end{enumerate}
Equivalently, $s\prec_\sigma t$ if, letting $j\ge1$ be the smallest such that $s_j \neq t_j$, either
$c:=|\{1\le i<j: s_i\in T^-_\sigma\}|$ is even and $s_j<t_j$, or $c$ is odd and $s_j>t_j$.
The signed shift is the map $\Sigma_\sigma:(\W_k,\prec_\sigma)\to (\W_k,\prec_\sigma)$ defined simply by $\Sigma_\sigma(s_1s_2s_3s_4\dots)=s_2s_3s_4\dots$.

To show that the two definitions of the signed shift as $\Sigma_\sigma$ and $\Sigma'_\sigma$ are order-isomorphic, consider the order-preserving bijection $\psi_\sigma:(\W_k,\prec_\sigma)\to(\W_k,\lex)$ that maps a word $s=s_1s_2s_3\dots$ to the word $a = a_1a_2a_3\dots$ where
$$a_i = \begin{cases} s_i & \mbox{if }|\{j< i : s_j\in T^-_\sigma\}| \mbox{ is even,} \\ \bar{s_i} & \mbox{if }|\{j< i : s_j\in T^-_\sigma\}| \mbox{ is odd.} \end{cases}$$
It is easy to check that $\psi_\sigma\circ\Sigma_\sigma = \Sigma_\sigma'\circ\psi_\sigma$, and so $\PP(\Sigma_\sigma)=\PP(\Sigma_\sigma')$.
Using the definition of the signed shift as $\Sigma_\sigma$, it is clear that its $n$-periodic points are the periodic words in $\W_k$ with period $n$, that is, words of the form
$s = (s_1s_2s_3\dots s_n)^\infty$ where $s_1s_2\dots s_n$ is {\em primitive} (sometimes called {\em aperiodic}), that is, not a concatenation of copies of a strictly shorter word.
We denote by $\W_{k,n}$ the set of periodic words in $\W_k$ with period $n$, and by $\WW_{k,n}$ the set of $n$-periodic orbits, where each orbit consists of the $n$ shifts of an $n$-periodic word.
For example, if $\sigma = +--$, then $s = (00110221)^\infty\in\W_{3,8}$ is an $8$-periodic point of $\Sigma_\sigma$, and $\PPat(s,\Sigma_\sigma)=12453786$.

If $\sigma = +^k$, then $\prec_\sigma$ is the lexicographic order $\lex$, and $\Sigma_\sigma$ is called the {\em $k$-shift}.
When $\sigma=-^k$, the map $\Sigma_\sigma$ is called the {\em reverse $k$-shift}. When $\sigma=+-$, the map $\Sigma_\sigma$ is the well-known {\em tent map}.

\subsection{Pattern avoidance and the $\sigma$-class}

Let $\S_n$ denote the set of permutations of $[n]$. We write permutations in one line notation as $\pi=\pi_1\pi_2\dots\pi_n\in\S_n$.
We say that $\tau \in \S_n$ {\em contains} $\rho \in \S_m$ if there exist indices $i_1<i_2<\dots <i_m$ such that $\st(\tau_{i_1}\tau_{i_2}\dots \tau_{i_m})=\rho_1\rho_2\dots \rho_m$.
Otherwise, we say that $\tau$ {\em avoids} $\rho$. We denote by $\Av(\rho)$ the set of permutations avoiding $\rho$, and we define $\Av(\rho^{(1)},\rho^{(2)},\dots)$ analogously as the set of permutations avoiding
all the patterns $\rho^{(1)},\rho^{(2)},\dots$.
A set of permutations $\mathscr{A}$ is called a {\it (permutation) class} if it is closed under pattern containment, that is, if a $\tau\in\mathscr{A}$ and $\tau$ contains $\rho$, then $\rho\in\mathscr{A}$.
Sets of the form $\Av(\rho^{(1)},\rho^{(2)},\dots)$ are permutation classes.
For example, $\Av(12)$ is the class of decreasing permutations.

Given classes $\mathscr{A}_0, \mathscr{A}_1, \dots, \mathscr{A}_{k-1}$, their juxtaposition, denoted $[\mathscr{A}_0 \, \mathscr{A}_1 \, \dots \, \mathscr{A}_{k-1}]$, is the set of permutations that can be expressed as concatenations $\alpha_0\alpha_1\dots\alpha_{k-1}$ where $\st(\alpha_t) \in \mathscr{A}_t$ for all $0\le t<k$. For example, $[\Av(21) \,  \Av(12)]$ is the set of {\em unimodal} permutations, i.e., those
$\pi\in\S_n$ satisfying $\pi_1<\pi_2<\dots<\pi_j>\pi_{j+1}>\dots>\pi_n$ for some $1\le j\le n$. The juxtaposition of permutation classes is again a class, and
as such, it can be characterized in terms of pattern avoidance. For example, $[\Av(21) \, \Av(12)] = \Av(213, 312)$.
Atkinson~\cite{Atk} showed that if  $\mathscr{A}_t$ can be characterized by avoidance of a finite set of patterns for each $t$,
then the same is true for $[\mathscr{A}_0 \, \mathscr{A}_1 \, \dots \, \mathscr{A}_{k-1}]$.

Let $\sigma=\sigma_0\sigma_1\dots\sigma_{k-1} \in \{+,-\}^k$ as before.
In order to characterize the periodic patterns of $\Sigma_\sigma$, we define a class $\S^\sigma$ of permutations, called the {\em $\sigma$-class},
which also appeared in~\cite{ARM,AAR}.
Let $\S^\sigma$ be the juxtaposition $[\mathscr{A}_0\, \, \mathscr{A}_1 \, \, \dots \, \, \mathscr{A}_{k-1}]$ where, for $0\le t<k$,
$$\mathscr{A}_t = \begin{cases} \Av(21) & \mbox{if }\sigma_t = +, \\ \Av(12) & \mbox{if }\sigma_t = -. \end{cases}$$
For example, $\S^{+-}$ is the class of unimodal permutations, and $\S^{++}$ is the class of {\em Grassmannian} permutations, i.e., those with at most one descent.
Let $\S_n^\sigma=\S^\sigma\cap\S_n$. Some examples of elements in $\S^{+--}= [\Av(21) \, \, \Av(12) \, \, \Av(12)]$ are the permutations $3\,5\,8\,9\,11\,7\,6\,1\,12\,10\,4\,2$ and $2\,5\,9\,10\,11\,8\,4\,3\,1\,12\,7\,6$, drawn in Figure~\ref{fig:perm}.
Note that the empty permutation belongs to $\Av(21)$ and to $\Av(12)$, so one trivially has $\S^{+--}\subset\S^{+-+-}$, for example.

If $\tau\in\S_n^\sigma$, a {\em $\sigma$-segmentation} of $\tau$ is a sequence $0=e_0\le e_1\le \dots \le e_k=n$ satisfying that
each segment $\tau_{e_{t}+1}\dots \tau_{e_{t+1}}$ is increasing if $t \in T^+_\sigma$ and decreasing if $t\in T^-_\sigma$. By definition, permutations in $\S^\sigma$ are precisely those that admit a $\sigma$-segmentation.

\begin{figure}[htb]
\centering
\begin{tikzpicture}[scale = .4]
\draw[help lines] (0,0) grid (12,12);
\draw [dotted] (0,0) -- (12,12);
\draw[red] (.5,.5) --(.5,2.5);
\draw[red] (.5,2.5) --(2.5,2.5);%1 to 3
\draw[red] (2.5,2.5) --(2.5,7.5);
\draw[red] (2.5,7.5) --(7.5,7.5);%3 to 8
\draw[red] (7.5,7.5) --(7.5,.5);
\draw[red] (7.5,.5) --(.5,.5);%8 to 1
\draw[red] (1.5,1.5) --(1.5,4.5);
\draw[red] (1.5,4.5) --(4.5,4.5);%2 to 5
\draw[red] (4.5,4.5) --(4.5,10.5);
\draw[red] (4.5,10.5) --(10.5,10.5);%5 to11
\draw[red] (10.5,10.5) --(10.5,3.5);
\draw[red] (10.5,3.5) --(3.5,3.5);%11 to 4
\draw[red] (3.5,3.5) --(3.5,8.5);
\draw[red] (3.5,8.5) --(8.5,8.5);%4 to 9
\draw[red] (8.5,8.5) --(8.5,11.5);
\draw[red] (8.5,11.5) --(11.5,11.5);%9 to 12
\draw[red] (11.5,11.5) --(11.5,1.5);
\draw[red] (11.5,1.5) --(1.5,1.5);%12 to 2
\draw[red] (5.5,5.5) --(5.5,6.5);
\draw[red] (5.5,6.5) --(6.5,6.5);%6 to 7
\draw[red] (6.5,6.5) --(6.5,5.5);
\draw[red] (6.5,5.5) --(5.5,5.5);%7 to 6
\draw [thick] (0,2) -- (1,3);
\draw [thick] (0,3) -- (1,2);%3
\draw [thick] (1,4) -- (2,5);
\draw [thick] (1,5) -- (2,4);%5
\draw [thick] (2,7) -- (3,8);
\draw [thick] (2,8) -- (3,7);%8
\draw [thick] (3,8) -- (4,9);
\draw [thick] (3,9) -- (4,8);%9
\draw [thick] (4,10) -- (5,11);
\draw [thick] (4,11) -- (5,10);%11
\draw [thick] (5,6) -- (6,7);
\draw [thick] (5,7) -- (6,6);%7
\draw [thick] (6,5) -- (7,6);
\draw [thick] (6,6) -- (7,5);%6
\draw [thick] (7,0) -- (8,1);
\draw [thick] (7,1) -- (8,0);%1
\draw [thick] (8,11) -- (9,12);
\draw [thick] (8,12) -- (9,11);%12
\draw [thick] (9,9) -- (10,10);
\draw [thick] (9,10) -- (10,9);%10
\draw [thick] (10,3) -- (11,4);
\draw [thick] (10,4) -- (11,3);%4
\draw [thick] (11,1) -- (12,2);
\draw [thick] (11,2) -- (12,1);%2
\end{tikzpicture}
\hspace{.1\linewidth}
\begin{tikzpicture}[scale = .4]
\draw[help lines] (0,0) grid (12,12);
\draw [dotted] (0,0) -- (12,12);
\draw[red] (.5,.5) --(.5,1.5);
\draw[red] (.5,1.5) --(1.5,1.5);
\draw[red] (1.5,1.5) --(1.5,4.5);
\draw[red] (1.5,4.5) --(4.5,4.5);
\draw[red] (4.5,4.5) --(4.5,10.5);
\draw[red] (4.5,10.5) --(10.5,10.5);
\draw[red] (10.5,10.5) --(10.5,6.5);
\draw[red] (10.5,6.5) --(6.5,6.5);
\draw[red] (6.5,6.5) --(6.5,3.5);
\draw[red] (6.5,3.5) --(3.5,3.5);
\draw[red] (3.5,3.5) --(3.5,9.5);
\draw[red] (3.5,9.5) --(9.5,9.5);
\draw[red] (9.5,9.5) --(9.5,11.5);
\draw[red] (9.5,11.5) --(11.5,11.5);
\draw[red] (11.5,11.5) --(11.5,5.5);
\draw[red] (11.5,5.5) --(5.5,5.5);
\draw[red] (5.5,5.5) --(5.5,7.5);
\draw[red] (5.5,7.5) --(7.5,7.5);
\draw[red] (7.5,7.5) --(7.5,2.5);
\draw[red] (7.5,2.5) --(2.5,2.5);
\draw[red] (2.5,2.5) --(2.5,8.5);
\draw[red] (2.5,8.5) --(8.5,8.5);
\draw[red] (8.5,8.5) --(8.5,.5);
\draw[red] (8.5,.5) --(.5,.5);
\draw [thick] (0,1) -- (1,2);
\draw [thick] (0,2) -- (1,1);
\draw [thick] (1,4) -- (2,5);
\draw [thick] (1,5) -- (2,4);
\draw [thick] (2,8) -- (3,9);
\draw [thick] (2,9) -- (3,8);
\draw [thick] (3,9) -- (4,10);
\draw [thick] (3,10) -- (4,9);
\draw [thick] (4,10) -- (5,11);
\draw [thick] (4,11) -- (5,10);
\draw [thick] (5,7) -- (6,8);
\draw [thick] (5,8) -- (6,7);
\draw [thick] (6,3) -- (7,4);
\draw [thick] (6,4) -- (7,3);
\draw [thick] (7,2) -- (8,3);
\draw [thick] (7,3) -- (8,2);
\draw [thick] (8,0) -- (9,1);
\draw [thick] (8,1) -- (9,0);
\draw [thick] (9,11) -- (10,12);
\draw [thick] (9,12) -- (10,11);
\draw [thick] (10,6) -- (11,7);
\draw [thick] (10,7) -- (11,6);
\draw [thick] (11,5) -- (12,6);
\draw [thick] (11,6) -- (12,5);
\end{tikzpicture}
\caption{Two permutations in $\S^\sigma$, where $\sigma=+--$, and their cycle structure. The permutation on the left has $\sigma$-segmentations $0\le 4\le 8\le 12$ and $0\le 5\le 8\le 12$.}
\label{fig:perm}
\end{figure}
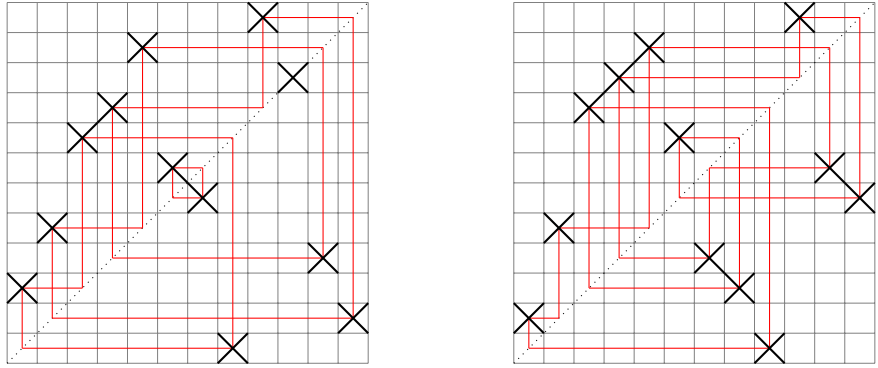

We denote by $\C_n$ (respectively, $\C^\sigma$, $\C_n^\sigma$) the set of cyclic permutations in $\S_n$ (respectively, $\S^\sigma$, $\S_n^\sigma$).
In Figure~\ref{fig:perm}, the permutation on the right is in $\C^\sigma$ while the permutation on the left is not.
It will be useful to define the map
$$\begin{array}{cccc}  \S_n & \to& \C_n\\
\pi &\mapsto &\hp,
\end{array}$$
where if $\pi = \pi_1\pi_2\dots \pi_n$ in one-line notation, then $\hp=(\pi_1,\pi_2,\dots,\pi_n)$ in cycle notation, that is, $\hp$ is the cyclic permutation that
sends $\pi_1$ to $\pi_2$, $\pi_2$ to $\pi_3$, and so on. Writing $\hp=\hp_1\hp_2\dots\hp_n$ in one-line notation, we have that $\hp_{\pi_i}=\pi_{i+1}$ for $1\le i\le n$, with the convention that $\pi_{n+1}:=\pi_1$.
For example, if $\pi = 17234856$, then $\pihat = 73486125$.

The map $\pi\mapsto\hp$ also plays an important role in~\cite{Elishifts}. Note that all the elements in the set $[\pi]$ of cyclic rotations of $\pi$ are mapped to the same element $\hp\in\C_n$.
Thus, letting $\SS_n=\{[\pi]:\pi\in \S_n\}$, this map induces a bijection $\theta$ between $\SS_n$ and $\C_n$, defined by $\theta([\pi])=\hp$.

\section{Description of periodic patterns of the signed shift}\label{sec:signed}

The main theorem of this paper is the following characterization of the periodic patterns of the signed shift $\Sigma_\sigma$, except in the case of the reverse shift. Throughout the paper we assume that $k\ge2$.

\begin{theorem}
\label{thm: description of periodic patterns}
Let $\sigma \in \{+,-\}^k$, $\sigma\neq-^k$. Then $\pi\in\PP(\Sigma_\sigma)$ if and only if $\hp\in \C^\sigma$.
\end{theorem}

An equivalent restatement of this theorem, whose proof will require a few lemmas, is that $\theta$ restricts to a bijection between $\PPP_n(\Sigma_\sigma)$ and $\C_n^\sigma$.
 The following lemma describes some conditions satisfied by the periodic patterns of $\Sigma_\sigma$, proving the forward direction of Theorem~\ref{thm: description of periodic patterns}.
For convenience, we will use the notation $\PPat_{\Sigma_\sigma}(s)$ or $\PPat_\sigma(s)$
instead of $\PPat(s,\Sigma_\sigma)$. Thus, for $\sigma \in \{+,-\}^k$, $\PPat_\sigma$ is a surjective map from $\W_{k,n}$ to $\PP(\Sigma_\sigma)$. If $\pi\in\S_n$, we will say that a word
$s_1s_2\dots s_n$ is {\em $\pi$-monotone} if $s_a\le s_b$ whenever $\pi_a<\pi_b$.

\begin{lemma} \label{lem: ppat}
Let $\sigma\in\{+,-\}^k$ be arbitrary, let $\pi\in\PP_n(\Sigma_\sigma)$, and let $s = (s_1\dots s_n)^\infty\in\W_{k,n}$ be such that $\pi=\PPat_\sigma(s)$. For $1\le t\le k$, let
$d_t=|\{i\in[n]:s_i< t\}|$, and let $d_0=0$. The following statements hold:
\begin{enumerate}
\item The word $s_1s_2\dots s_n$ is {\em $\pi$-monotone}, that is, $s_i=t$ if and only if $d_{t} < \pi_i\le d_{t+1}$.
\item If $d_{t} < \pi_i<\pi_j\leq d_{t+1}$, then $\pi_{i+1}<\pi_{j+1}$ if $t \in T^+_\sigma$, and $\pi_{i+1}>\pi_{j+1}$ if $t \in T^-_\sigma$,
where we let $\pi_{n+1}:=\pi_1$.
\item The sequence $d_0,d_1,\dots,d_k$ is a $\sigma$-segmentation of $\hp$. In particular, $\hp \in \C^\sigma$.
\end{enumerate}
\end{lemma}

\begin{proof}
Since $\PPat_\sigma(s)=\pi$, it is clear for all $a,b\in[n]$, $\pi_a<\pi_b$ implies $s_a\le s_b$, from where part 1 follows.
To prove part 2, suppose that $d_{t} < \pi_i<\pi_j\leq d_{t+1}$, and so $s_{[i, \infty)} \prec_\sigma s_{[j,\infty)}$. By part 1, we have $s_i = s_j=t$.
If $t \in T^+_\sigma$, then $s_{[i+1, \infty)} \prec_\sigma s_{[j+1,\infty)}$, and so $\pi_{i+1}<\pi_{j+1}$.
Similarly, if $t \in T^-_\sigma$, then $s_{[j+1, \infty)} \prec_\sigma s_{[i+1,\infty)}$, and so $\pi_{i+1}>\pi_{j+1}$.

Now let $0\le t<k$, and suppose that the indices $j$ such that $s_j=t$ are $j_1,\dots, j_m$, ordered in such a way that $\pi_{j_1}< \pi_{j_2}<\dots< \pi_{j_m}$,
where $m=d_{t+1}-d_{t}$. Then part 1 implies that $\pi_{j_\ell}=d_{t-1}+\ell$ for $1\le\ell\le m$,
and part 2 implies that $\pi_{j_1+1}< \pi_{j_2+1}<\dots< \pi_{j_m+1}$ if $t\in T^+_\sigma$,
and $\pi_{j_1+1}> \pi_{j_2+1}>\dots> \pi_{j_m+1}$ if $t\in T^-_\sigma$.
Using that $\pi_{j_\ell+1}=\hp_{\pi_{j_\ell}}=\hp_{d_{t}+\ell}$, this is equivalent to
$\hp_{d_{t}+1}< \hp_{d_{t}+2}<\dots<\hp_{d_{t}+m}$ if $t\in T^+_\sigma$,
and $\hp_{d_{t}+1}>\hp_{d_{t}+2}>\dots>\hp_{d_{t}+m}$ if $t\in T^-_\sigma$. Note that $d_{t}+m = d_{t+1}$, so this condition states that $d_0,d_1,\dots,d_k$ is a $\sigma$-segmentation of $\hp$.
Since $\hp$ is a cyclic permutation, this proves that $\hp\in\C^\sigma$.
\end{proof}

The next three lemmas will be used in the proof of the backward direction of Theorem~\ref{thm: description of periodic patterns}.

\begin{lemma}\label{lem: Csigma}
Let $\sigma\in\{+,-\}^k$ be arbitrary, let $\pi\in\S_n$, and suppose that
$0=e_0\le e_1\le \dots \le e_k=n$ is a $\sigma$-segmentation of $\hp$.
Suppose that $e_{t} < \pi_i<\pi_j\leq e_{t+1}$ for some $1\leq i,j\leq n$. Then $\pi_{i+1}<\pi_{j+1}$ if $t \in T^+_\sigma$, and $\pi_{i+1}>\pi_{j+1}$ if $t \in T^-_\sigma$,
where we let $\pi_{n+1}:=\pi_1$.
\end{lemma}

\begin{proof}
Since $e_{t} < \pi_i<\pi_j\leq e_{t+1}$, both $\hp_{\pi_i}$ and $\hp_{\pi_j}$ lie in the segment $\hp_{e_{t}+1}\dots \hp_{e_{t+1}}$. If $t \in T^+_\sigma$, this segment is increasing,
so $\pi_{i+1}=\hp_{\pi_i}<\hp_{\pi_j}=\pi_{j+1}$. The argument is analogous if $t \in T^-_\sigma$.
\end{proof}

It is important to note that for any $\pi\in\S_n$ and any sequence $0=e_0\le e_1\le \dots \le e_k=n$, there is a unique $\pi$-monotone word $s_1s_2\dots s_n$ such that $|\{i\in[n]:s_i< t\}|=e_t$ for every $t$. This word is defined by letting $s_i = t$ whenever $e_{t}< \pi_i\leq e_{t+1}$.
We say that $s_1s_2\dots s_n$ is the {\em $\pi$-monotone word induced by} $e_0,e_1,\dots,e_k$.

\begin{lemma}\label{lem:sisj}
Let $\sigma \in \{+,-\}^k$ be arbitrary, and let $\pi\in\S_n$ be such that $\hp \in \C^\sigma$.
Take any $\sigma$-segmentation of $\hp$, and let $s_1\dots s_n$ be the $\pi$-monotone word induced by it. Let $s=(s_1s_2\dots s_n)^\infty$.
If $1\le i,j\le n$ are such that $\pi_i<\pi_j$, then either $s_{[i,\infty)}=s_{[j,\infty)}$ or $s_{[i,\infty)} \prec_\sigma s_{[j,\infty)}$.
In particular, if $s_1\dots s_n$ is primitive, then $\PPat_\sigma(s)=\pi$.
\end{lemma}

\begin{proof}
Suppose that $\pi_i<\pi_j$. If $s_{[i,\infty)}\neq s_{[j,\infty)}$, let $a\ge0$ be the smallest such that $s_{i+a}\neq  s_{j+a}$,
and let $h=|\{0\le \ell \le a-1 : s_\ell \in T^-_\sigma\}|$.
 If $h$ is even, then Lemma \ref{lem: Csigma} applied $a$ times shows that $\pi_{i+a}<\pi_{j+a}$. Since $s_{i+a}\neq s_{j+a}$, we must then have $s_{i+a} < s_{j+a}$, because of $\pi$-monotonicity.
Thus, $s_{[i,\infty)}\prec_\sigma s_{[j,\infty)}$ by definition of $\prec_\sigma$, since the word $s_is_{i+1}\dots s_{i+a-1}= s_js_{j+1}\dots s_{j+a-1}$ has an even number of letters in $T^-_\sigma$.
 Similarly, if $h$ is odd, then Lemma \ref{lem: Csigma} shows that  $\pi_{i+a}>\pi_{j+a}$. Since $s_{i+a}\neq s_{j+a}$, we must have $s_{i+a} > s_{j+a}$, and thus $s_{[i,\infty)}\prec_\sigma s_{[j,\infty)}$
by definition of $\prec_\sigma$.

If $s_1\dots s_n$ is primitive, the case $s_{[i,\infty)}=s_{[j,\infty)}$ can never occur when $i\neq j$, and so $\pi_i<\pi_j$ if and only if $s_{[i,\infty)} \prec_\sigma s_{[j,\infty)}$.
It follows that $\PPat_\sigma(s)=\pi$.
\end{proof}

\begin{lemma}\label{lem: s1sn}
Let $\sigma \in \{+,-\}^k$ be arbitrary. If $\sigma=-^k$, additionally assume that $n \not=2\bmod 4$.
Let $\pi\in\S_n$ be such that $\hp \in \C^\sigma$. Then there exists a $\sigma$-segmentation of $\hp$ such that the $\pi$-monotone word $s_1\dots s_n$ induced by it is primitive, and the word $s=(s_1\dots s_n)^\infty$ satisfies $\PPat_\sigma(s)=\pi$.
 Furthermore, if $\sigma=+^k$ or $\sigma=-^k$, then {\em every} $\sigma$-segmentation of $\hp$ has this property.
\end{lemma}

\begin{proof}
Since $\hp \in \C^\sigma$, it admits some $\sigma$-segmentation. Pick one, say $0=e_0\le e_1\le \dots \le e_k=n$, and let $s_1\dots s_n$ be the $\pi$-monotone word induced by it. In this proof we take the indices of $\pi$ mod $n$, that is, we define $\pi_{i+jn}=\pi_i$ for $i\in[n]$.

Suppose that $s_1\dots s_n$ is not primitive, so it can be written as $q^m$ for some $m\ge 2$ and some primitive word $q$ with $|q|=r=n/m$.
Then, $s_i = s_{i+r}$ for all $i$. Let $g = |\{i\in[r] :  s_i \in T^-_\sigma\}|$.
Fix $i$, and let $t=s_i = s_{i+r}$. Because of the way that $s_1\dots s_n$ is defined, we must have $e_t< \pi_i,\pi_{i+r}\leq e_{t+1}$, so we can apply Lemma~\ref{lem: Csigma} to this pair.

Suppose first that $g$ is even.
If $\pi_i<\pi_{i+r}$, then applying Lemma~\ref{lem: Csigma} $r$ times we get $\pi_{i+r}<\pi_{i+2r}$,
since the inequality involving $\pi_{i+\ell}$ and $\pi_{i+r+\ell}$ switches exactly $g$ times as $\ell$ increases from $0$ to $r$.  Starting with $i=1$ and applying this argument repeatedly, we see that if $\pi_1<\pi_{1+r}$, then
$\pi_1<\pi_{1+r}<\pi_{1+2r}<\dots <\pi_{1+(m-1)r}<\pi_{1+mr}=\pi_1$, which is a contradiction.
A symmetric argument shows that if $\pi_1>\pi_{1+r}$, then $\pi_1>\pi_{1+r}>\pi_{1+2r}>\dots >\pi_{1+(m-1)r}>\pi_{1+mr}=\pi_1$.

It remains to consider the case that $g$ is odd. If $m$ is even and $m\ge4$, then letting $q' = qq$ we have $s_1s_2\dots s_n = (q')^{\frac{m}{2}}$.
Letting $r'=|q'|=2r$ and $g'= |\{i\in[2r] :  s_i \in T^-_\sigma\}|=2g$, the same argument as above using $r'$ and $g'$ yields a contradiction.
If $m$ is odd, suppose without loss of generality that $\pi_1<\pi_{1+r}$.
Note that applying Lemma~\ref{lem: Csigma} $r$ times to the inequality $\pi_i<\pi_{i+r}$ (respectively $\pi_i>\pi_{i+r}$) yields $\pi_{i+r}>\pi_{i+2r}$ (respectively $\pi_{i+r}<\pi_{i+2r}$) in this case, since the inequality involving $\pi_{i+\ell}$ and $\pi_{i+r+\ell}$ switches an odd number of times.
Consider two cases:
\begin{itemize}
\item If $\pi_{1}<\pi_{1+2r}$, then Lemma~\ref{lem: Csigma} applied repeatedly in blocks of $2r$ times yields $\pi_1<\pi_{1+2r}<\pi_{1+4r}<\dots< \pi_{1+(m-1)r}$.
Applying now Lemma~\ref{lem: Csigma} $r$ times starting with $\pi_1<\pi_{1+(m-1)r}$ gives $\pi_{1+r}>\pi_{1+mr}=\pi_1$, which contradicts the assumption $\pi_1<\pi_{1+r}$.

\item If $\pi_{1}>\pi_{1+2r}$, applying Lemma~\ref{lem: Csigma} $r$ times we get $\pi_{1+r}<\pi_{1+3r}$, and by repeated application of the lemma in blocks of $2r$ times it follows that
$\pi_{1+r}<\pi_{1+3r}<\pi_{1+5r}<\dots< \pi_{1+(m-2)r}<\pi_{1+mr}=\pi_1$, contradicting again the assumption $\pi_1<\pi_{1+r}$.
\end{itemize}

The only case left is when $g$ is odd and $m = 2$, that is, when $s_1s_2\dots s_n = q^2$ and $q$ has an odd number of letters in $T^-_\sigma$.
Note that this situation does not happen when $\sigma = +^k$ (since in this case $g=0$) and, although it can happen when $\sigma=-^k$, in this case we would have that $T^-_\sigma=\{0,1,\dots,k-1\}$, and so $n = 2r = 2g= 2 \bmod 4$, which we are excluding in the statement of the theorem.

Thus, we can assume that there exists some $1\le\ell<k$ such that $\sigma_{\ell-1}\sigma_{\ell}$ is either $+-$ or $-+$.
We will show that there is a $\sigma$-segmentation $0=e'_0\le e'_1\le \dots \le e'_k=n$ of $\hp$ such that the $\pi$-monotone word $s'_1s'_2\dots s'_n$ induced by it is primitive.

Suppose that $\sigma_{\ell-1}\sigma_{\ell}=+-$ (the case $\sigma_{\ell-1}\sigma_{\ell}=-+$ is very similar). Then $\hp_{e_{\ell-1}+1}<\dots <\hp_{e_{\ell}}$ and $\hp_{e_{\ell}+1}>\dots >\hp_{e_{\ell+1}}$.
If $\hp_{e_\ell}<\hp_{e_\ell+1}$ (respectively, $\hp_{e_\ell}>\hp_{e_\ell+1}$), let $e'_\ell := e_\ell+1$ (respectively, $e'_\ell := e_\ell-1$), and $e'_t:=e_t$ for all $t\neq \ell$.
Clearly $e'_0,e'_1,\dots,e'_k$ is a $\sigma$-segmentation of $\hp$. Its induced $\pi$-monotone word $s'_1\dots s'_n$ differs from the original word $s_1s_2\dots s_n=q^2$ by one entry,
making the number of $\ell$s that appear in  $s'_1\dots s'_n$ be odd instead of even. Thus, $s'_1\dots s'_n$ can no longer be written as $(q')^2$ for any $q'$, so it is primitive by the above argument.

We have shown that there is always a $\sigma$-segmentation of $\hp$ such that the $\pi$-monotone word induced by it is primitive. Denote this word by $s_1\dots s_n$, and let $s=(s_1\dots s_n)^\infty$.
Now Lemma~\ref{lem:sisj} shows that $\PPat_\sigma(s)=\pi$.
\end{proof}

We can now combine the above lemmas to prove our main theorem.

\begin{proof}[Proof of Theorem \ref{thm: description of periodic patterns}]
If $\pi\in\PP(\Sigma_\sigma)$, then $\hp\in \C^\sigma$ by part 3 of Lemma \ref{lem: ppat}.
Conversely, $\pi\in\S_n$ is such that $\hp\in \C^\sigma$, then the word $s$ constructed in Lemma~\ref{lem: s1sn} satisfies $\PPat_\sigma(s)=\pi$, and so $\pi\in\PP(\Sigma_\sigma)$.
\end{proof}

For $\sigma=-^k$, the same proof yields the following weaker result.

\begin{proposition} \label{prop: description reverse}
Let $\sigma = -^k$. If $\pi\in\PP_n(\Sigma_\sigma)$, then $\hp\in \C_n^\sigma$. Additionally, the converse holds if $n\neq 2\bmod 4$.
\end{proposition}

In Section~\ref{sec:reverse2odd}, when we enumerate the periodic patterns of the reverse shift, we will see exactly for how many patterns the converse fails when $n\neq 2\bmod 4$.

We end this section describing the effect (also mentioned in~\cite{Amigosigned}) of reversing $\sigma$ on the allowed patterns of $\Sigma_\sigma$. Define the reversal of $\sigma=\sigma_0\sigma_1\dots\sigma_{k-1}$ to be $\sigma^R = \sigma_{k-1}\dots \sigma_1\sigma_0$. If $\pi\in\S_n$, then the complement of $\pi$ is the permutation $\pi^c$ where $\pi^c_i = n+1-\pi_i$ for $1\le i\le n$.

\begin{proposition} \label{prop: pi complement}
For every $\sigma\in\{+,-\}^k$,
$\pi\in\Allow(\Sigma_\sigma)$ if and only if $\pi^c\in\Allow(\Sigma_{\sigma^R})$, and $\pi\in\PP(\Sigma_\sigma)$ if and only if $\pi^c\in\PP(\Sigma_{\sigma^R})$.
\end{proposition}

\begin{proof}
Given $s\in\W_k$, let $\tilde{s}\in\W_k$ be the word whose $i$th letter is $\tilde{s_i} = k-1-s_i$ for all $i$. For the first statement, it is enough to show that $\Pat(s,\Sigma_\sigma,n)=\pi$ if and only if $\Pat(\tilde{s},\Sigma_{\sigma^R})=\pi^c$.
This will follow if we show that for words $s$ and $t$, $s\prec_\sigma t$ if and only if $\tilde{t}\prec_{\sigma^R} \tilde{s}$.
To prove this claim, let $j$ be the first position where $s$ and $t$ differ. Then $j$ is also the first position where $\tilde{s}$ and $\tilde{t}$ differ, and $s_j<t_j$ if and only if $\tilde{s_j}>\tilde{t_j}$.
By definition, $s_i \in T^-_\sigma$ if and only if  $\tilde{s_i} \in T^-_{\sigma^R}$, from where the claim follows.

Since $s$ is periodic of period $n$ if and only if $\tilde{s}$ is as well, the result for periodic patterns holds as well.
\end{proof}

Proposition~\ref{prop: pi complement} can be stated more generally for any function $f:[0,1]\to[0,1]$ as follows. If we define $g:[0,1]\to[0,1]$ by $g(x)=1-f(1-x)$, we see by induction on $i$ that $g^i(x)=1-f^i(1-x)$ for $i\ge1$.
It follows that $\Pat(x,g,n)=\Pat(1-x,f,n)^c$ for all $x\in[0,1]$, and so $f$ and $g$ have the same number of allowed and periodic patterns.

\section{Enumeration for special cases}\label{sec:specialcases}

For particular values of $\sigma$, we can give a formula for the number of periodic patterns of $\Sigma_\sigma$. This is the case when $\sigma= +-$ (equivalently, when $\sigma=-+$), $\sigma = +^k$, and $\sigma =-^k$, for any $k\ge2$.
Recall that $p_n(\Sigma_\sigma)=|\PPP_n(\Sigma_\sigma)|$ is the number of equivalence classes $[\pi]$ of periodic patterns of $\Sigma_\sigma$.

\begin{lemma}\label{lem:periodic orbits}
For any $\sigma\in\{+,-\}^k$ with $k\geq 2$, the number of $n$-periodic orbits of $\Sigma_\sigma$ is
\begin{equation}\label{eq:Lk} |\WW_{k,n}|=L_k(n) = \frac{1}{n}\sum_{d|n} \mu(d)k^{\frac{n}{d}},\end{equation}
where $\mu$ denotes the M\"obius function.
\end{lemma}

\begin{proof}
Recall that $\WW_{k,n}$ is the set of $n$-periodic orbits of $\Sigma_\sigma$, and that each such orbit consists of the $n$ shifts of an $n$-periodic point $s=(s_1s_2\dots s_n)^\infty\in\W_{k,n}$, where $s_1s_2\dots s_n$ is primitive.
Thus, $|\WW_{k,n}|$ is the number of primitive words of length $n$ over a $k$-letter alphabet up to cyclic rotation. In different terminology, this is the number of $n$-bead primitive necklaces with $k$ colors, which is well-known to equal the formula in the statement.
\end{proof}

In some of the proofs in this section, it will be useful to refer to the diagram in Figure~\ref{fig:maps}, which summarizes some of the maps involved.
If $\sigma\in\{+,-\}^k$, we denote by $\PPPat_\sigma$ (or  $\PPPat_{\Sigma_\sigma}$) the map from
$\WW_{k,n}$ to $\PPP_n(\Sigma_\sigma)$ that sends the periodic orbit of $s\in\W_{k,n}$ to the equivalence class $[\pi]$, where $\pi=\PPat_\sigma(s)$. This map is clearly well defined.

\begin{figure}[htb]
\centering
\begin{tikzpicture}
  \matrix (m) [matrix of math nodes,row sep=2.5em,column sep=4em,minimum width=2em]
  {
     \W_{k,n} & \PP_n(\Sigma_\sigma) &  \\
     \WW_{k,n} & \PPP_n(\Sigma_\sigma) & \C_n^\sigma \\
     };
  \path[-stealth]
    (m-1-1) edge [dashed] (m-2-1)
            edge node [above] {\footnotesize $\PPat_\sigma$} (m-1-2)
    (m-1-2) edge [dashed] (m-2-2)
            edge node [above right] {$\pi\mapsto\hp$} (m-2-3)
    (m-2-1) edge node [above] {\footnotesize $\PPPat_\sigma$} (m-2-2)
    (m-2-2) edge node [above] {$\theta$} (m-2-3);
\end{tikzpicture}
\caption{Some maps used in Section~\ref{sec:specialcases}. The map from $\W_{k,n}$ to $\WW_{k,n}$ maps each periodic point to its corresponding periodic orbit, and the map from $\PP_n(\Sigma_\sigma)$ to $\PPP_n(\Sigma_\sigma)$ maps each permutation $\pi$ to its equivalence class $[\pi]$.
The map $\PPat_\sigma$ (equivalently, $\PPPat_\sigma$) is surjective by definition, but in general not injective.
By Theorem~\ref{thm: description of periodic patterns} and Proposition~\ref{prop: description reverse}, the map $\theta$ restricts to a bijection between $\PPP_n(\Sigma_\sigma)$ and $\C_n^\sigma$, unless $\sigma=-^k$ and $n=2\bmod4$.}
\label{fig:maps}
\end{figure}

The following lemmas will also be needed when counting periodic patterns of shifts and reverse shifts.

\begin{lemma}\label{lem:iff}
Suppose that $\sigma=+^k$, or that $\sigma=-^k$ and $n\neq2\bmod4$. Let $\pi\in\PP_n(\Sigma_\sigma)$ (equivalently, by Theorem~\ref{thm: description of periodic patterns} and Proposition~\ref{prop: description reverse}, $\hp\in\C_n^\sigma$), and let $s=(s_1\dots s_n)^\infty\in\W_{k,n}$.
Then $s$ satisfies $\PPat_\sigma(s)=\pi$ if and only if $s_1\dots s_n$ is the $\pi$-monotone word induced by some $\sigma$-segmentation of $\hp$.
\end{lemma}

\begin{proof}
If $\PPat_\sigma(s)=\pi$, then by Lemma~\ref{lem: ppat}, $s_1\dots s_n$ is the $\pi$-monotone word induced by the $\sigma$-segmentation $d_0,d_1,\dots,d_k$, where $d_t=|\{i\in[n]:s_i< t\}|$.
Conversely, given any $\sigma$-segmentation $e_0,e_1,\dots,e_k$ of $\hp$, then the last statement in Lemma~\ref{lem: s1sn} implies that the $\pi$-monotone word $s_1\dots s_n$ induced by it satisfies $\PPat_\sigma((s_1\dots s_n)^\infty)=\pi$.
\end{proof}

\subsection{The tent map}

We denote the tent map by $\tent=\Sigma_{+-}$. Recall that $\tent$ is order-isomorphic to the map on the unit interval whose graph appears on the left of Figure~\ref{fig:maps}.
The characterization of the periodic patterns of $\tent$ follows from Theorem \ref{thm: description of periodic patterns}, noticing that permutations in the class $[\Av(21) \, \, \Av(12)]$ are precisely unimodal permutations.

\begin{corollary}\label{cor:PPtent}
$\pi\in\PP(\tent)$ if and only if $\hp$ is unimodal.
\end{corollary}

The proof of the next theorem, which gives a formula for the number of periodic patterns of the tent map,
is closely related to a construction of Weiss and Rogers~\cite{WR} aimed at counting unimodal cyclic permutations, even though
the notion of periodic patterns of a map had not yet been developed.

\begin{theorem}\label{thm:enumtent}
$$p_n(\tent) =  \frac{1}{2n} \sum_{\substack{ d|n \\ d\ \mathrm{odd}}} \mu(d) 2^{\frac{n}{d}}.$$
\end{theorem}

\begin{proof}
Let $\O_n$ be the set of binary words $s = (s_1\dots s_n)^\infty\in\W_{2,n}$ where the primitive word $s_1\dots s_n$ has an odd number of ones.
We will show that the map $\PPat_\tent:\W_{2,n}\rightarrow\PP_n(\tent)$ restricts a bijection between $\O_n$ and $\PP_n(\tent)$.
We will prove that for each $\pi\in\PP_n(\tent)$ there are either one or two
elements $s\in\W_{2,n}$ such that $\PPat_\tent(s)=\pi$, and that exactly one of them is in $\O_n$.

Fix $\pi\in\PP_n(\tent)$, and recall from Corollary~\ref{cor:PPtent} that $\hp_1<\hp_2<\dots <\hp_m>\hp_{m+1}>\dots > \hp_n$ for some $m$.
Let $s = (s_1\dots s_n)^\infty\in\W_{2,n}$ be such that $\PPat_\tent(s)=\pi$, and let $d=|\{1\le i\le n: s_i=0\}|$.
By part 3 of Lemma~\ref{lem: ppat}, we have $\hp_1<\hp_2<\dots <\hp_d$ and $\hp_{d+1}>\hp_{d+2}>\dots >\hp_n$.
It follows that $d=m$ or $d=m-1$, corresponding to the two $+-$-segmentations of $\hp$.
Since $s_1\dots s_n$ is $\pi$-monotone by part 1 of Lemma~\ref{lem: ppat}, the words corresponding to these two possibilities for $d$ differ only in one position (specifically, in $s_j$, where $j$ is such that $\pi_j=m$).
In general, as shown in the proof of Lemma~\ref{lem: s1sn}, the word $s_1\dots s_n$ induced by a $+-$-segmentation of $\hp$ may not be primitive. However, this can only happen when $n$ is even and $s_1\dots s_n = q^2$,
in which case $s_1\dots s_n$ would have an even number of ones.
Thus, among the two possibilities for $d$, the one in which $s_1\dots s_n$ has an odd number of ones guarantees that this word is primitive, and so $s\in\O_n$ and $\PPat_\tent(s)=\pi$.

To find $|\O_n|$ we use the M\"obius inversion formula. Let $a(n)$ be the number of primitive binary words of length $n$ with an odd number of ones, and note that $a(n)=|\O_n|=|\PP_n(\tent)|$.
Let $b(n)$ be the number of binary words (not necessarily primitive) of length $n$ with an odd number of ones. It is well known that $b(n)=2^{n-1}$. Every such word can be written as $s=q^d$,
where $q$ is primitive and has an odd number of ones, and $d$ is an odd number that divides $n$. It follows that
$$b(n)=\sum_{\substack{d|n\\ d\ \mathrm{odd}}}a(n/d)=\sum_{\substack{d|n\\ n/d\ \mathrm{odd}}}a(d).$$
Writing $n=2^r m$, where $m$ is odd, and letting $b_r(m)=b(2^r m)$ and $a_r(m)=a(2^r m)$, the above formula becomes
$$b_r(m)=b(2^r m)=\sum_{d|m}a(2^r d)=\sum_{d|m}a_r(d).$$
Thus, by M\"obius inversion,
$$a_r(m)=\sum_{d|m}\mu(d)b_r(m/d),$$
which is equivalent to
$$a(n)=\sum_{d|m}\mu(d)b(n/d)=\sum_{\substack{d|n\\ d\ \mathrm{odd}}}\mu(d)b(n/d)=\sum_{\substack{d|n\\ d\ \mathrm{odd}}}\mu(d)2^{n/d-1}.$$
It follows that
$$p_n(\tent) = \frac{1}{n}a(n)=\frac{1}{2n}\sum_{\substack{d|n\\ d\ \mathrm{odd}}}\mu(d)2^{n/d}.$$
\end{proof}

The proof of Theorem \ref{thm:enumtent}, in combination with Corollary \ref{cor:PPtent}, shows that the map $\theta\circ\PPPat_\tent: \WW_{2,n} \to \C^{+-}_n$ is a bijection between primitive binary necklaces (words up to cyclic rotation) with an odd number of ones and unimodal cyclic permutations.

\begin{corollary}\label{cor:-+}
$$p_n(\Sigma_{-+}) =  \frac{1}{2n} \sum_{\substack{ d|n \\ d\ \mathrm{odd} }} \mu(d) 2^{\frac{n}{d}}.$$
\end{corollary}

\begin{proof}
This follows immediately from Proposition \ref{prop: pi complement} and Theorem \ref{thm:enumtent}.
\end{proof}

\subsection{The $k$-shift}

Recall that the $k$-shift is the map $\Sigma_\sigma$ where $\sigma =+^k$. Let us denote this map by $\Sigma_k$ for convenience. The second picture in Figure~\ref{fig:maps} shows the graph of a map on the unit interval that is order-isomorphic to $\Sigma_3$.
The allowed patterns of the $k$-shift were characterized and enumerated by Elizalde~\cite{Elishifts}, building up on work by Amig\'o {\it et al.}~\cite{AEK}.

In this section we describe and enumerate the periodic patterns of the $k$-shift. Denote the descent set of $\pi\in\S_n$ by $\Des(\pi)=\{i\in[n-1]:\pi_i>\pi_{i+1}\}$, and by $\des(\pi)=|\Des(\pi)|$ the number of descents of $\pi$.
In the case of the $k$-shift, Theorem \ref{thm: description of periodic patterns} states that
$\pi\in\PP(\Sigma_k)$ if and only if $\hp$ is a cyclic permutation that can be written as a concatenation of $k$ increasing sequences. The following corollary follows from this description.

\begin{corollary}\label{cor:PPshift}
$\pi\in\PP(\Sigma_k)$ if and only if $\des(\hp)\le k-1$.
\end{corollary}

In other words, $\theta$ gives a bijection between $\PPP_n(\Sigma_k)$ and permutations in $\C_n$ with at most $k-1$ descents.
It will be convenient to define, for $1\le i\le n$,
$$C(n,i)=|\{\tau\in\C_n:\des(\tau)=i-1\}|.$$
We start by giving a formula for the number of periodic patterns of the binary shift. Recall the formula for $L_k(n)$ given in Eq.~\eqref{eq:Lk}.

\begin{theorem}\label{thm:enum2shift}
For $n\ge2$,
$$p_n(\Sigma_2)=C(n,2)=L_2(n).$$
\end{theorem}

\begin{proof}
When $n\ge2$, there are no permutations in $\C_n$ with no descents. By Proposition~\ref{prop: description reverse}, the map $\theta$ gives a bijection between $\PPP(\Sigma_2)$ and $\C_n^{++}=\{\tau\in\C_n:\des(\tau)=1\}$, so $p_n(\Sigma_2)=C(n,2)$.

Next we show that the map $\PPat_{\Sigma_2}:\W_{2,n}\to\PP_n(\Sigma_2)$ is a bijection, and therefore so
is the map $\PPPat_{\Sigma_2}:\WW_{2,n}\to\PPP_n(\Sigma_2)$ (see Figure~\ref{fig:maps}), which implies that $p_n(\Sigma_2)=L_2(n)$ by Lemma~\ref{lem:periodic orbits}.
By definition, the map $\PPat_{\Sigma_2}:\W_{2,n}\to\PP_n(\Sigma_2)$ is surjective, so we just need to show that it is injective as well.
Let $\pi\in\PP_n(\Sigma_2)$, and let $s=(s_1\dots s_n)^\infty\in\W_{2,n}$ be such that $\PPat_{\Sigma_2}(s) = \pi$.
By Lemma~\ref{lem:iff}, $s_1\dots s_n$ is the $\pi$-monotone word induced by some $++$-segmentation of $\hp$. Since $\hp$ has one descent, there is only one such segmentation, so $s$ is uniquely determined.
\end{proof}

For $k\ge3$, we have the following result.

\begin{theorem}\label{thm: enum shift}
For $k\ge3$ and $n\ge2$,
$$p_n(\Sigma_k)-p_n(\Sigma_{k-1})=C(n,k) = L_k(n) - \sum_{i = 2}^{k-1} \binom{n+k-i}{k-i}C(n,i).$$
\end{theorem}

\begin{proof}
By Corollary~\ref{cor:PPshift}, $\pi\in\PP_n(\Sigma_k)\setminus\PP_n(\Sigma_{k-1})$ if and only if $\des(\hp)=k-1$, and so $\theta$ is a bijection between $\PPP_n(\Sigma_k)\setminus\PPP_n(\Sigma_{k-1})$
and $n$-cycles with $k-1$ descents, from where $C(n,k)=p_n(\Sigma_k)-p_n(\Sigma_{k-1})$.

To prove the recursive formula for $C(n,k)$, we find the cardinality of $\WW_{k,n}$ in two ways. On one hand, this number equals $L_k(n)$ by Lemma~\ref{lem:periodic orbits}.
On the other hand, consider the map $\PPPat_{\Sigma_k}:\WW_{k,n}\to\PPP_n(\Sigma_k)$, which is surjective, but in general not injective. We can obtain $|\WW_{k,n}|$ by adding the cardinalities of the preimages of the elements of $\PPP_n(\Sigma_k)$ under this map.

For fixed $\pi\in\PP_n(\Sigma_k)$, let us count how many words $s\in\W_{k,n}$ satisfy $\PPat_{\Sigma_k}(s)=\pi$ (equivalently, how many orbits in $\WW_{k,n}$ are mapped to $[\pi]$ by $\PPPat_{\Sigma_k}$). By Lemma~\ref{lem:iff},
this number is equal to the number of $+^k$-segmentations of $\hp$.
If $\des(\hp)=i-1$, it is a simple exercise to show that there are $\binom{n+k-i}{k-i}$ such segmentations $0=e_0\le e_1\le \dots \le e_k=n$, since $\Des(\hp)$ has to be a subset of $\{e_1,\dots,e_{k-1}\}$.

By Corollary~\ref{cor:PPshift}, for each $2\le i\le k$, the number of equivalence classes $[\pi]\in\PPP_n(\Sigma_k)$ where $\des(\hp)=i-1$ is $C(n,i)$.
It follows that $$L_k(n)=\sum_{i=2}^k \binom{n+k-i}{k-i} C(n,i),$$
which is equivalent to the stated formula.
\end{proof}

It is clear from Theorem~\ref{thm: enum shift} that for $n\ge2$,
$$p_n(\Sigma_k)= \sum_{i=2}^k C(n,i).$$
Let us show an example that illustrates how, in the above proof, the words $s\in\W_{k,n}$ with $\PPat_{\Sigma_k}(s)=\pi$ are constructed for given $\pi$.
Let $k=5$, and let $\pi = 165398427\in\PP_9(\Sigma_5)$. Then $\hp = 679235148$, which has descent set $\Des(\hp)= \{3, 6\}$.
The $+^5$-segmentation with $e_1=3$, $e_2=6$, $e_3=e_4=9$ induces the $\pi$-monotone word $s_1\dots s_9=011022102$. The $+^5$-segmentation with $e_1=2$, $e_2=3$, $e_3=6$, $e_4=7$ induces $s_1\dots s_9 = 022144203$.

We conclude by mentioning that the second equality in Theorem \ref{thm: enum shift} also follows from a result of Gessel and Reutenauer \cite[Theorem 6.1]{GesReut}, which is proved using quasi-symmetric functions.

\subsection{The reverse $k$-shift, when $n \neq 2 \mbox{ mod } 4$}
\label{sec:reverse}

The reverse $k$-shift is the map $\Sigma_\sigma$ where $\sigma =-^k$. We denote this map by $\rs_k$ in this section.  The third picture in Figure~\ref{fig:maps} shows the graph of a map on the unit interval that is order-isomorphic to $\rs_4$.
Denote the ascent set of $\pi\in\S_n$ by $\Asc(\pi)=\{i\in[n-1]:\pi_i<\pi_{i+1}\}$, and by $\asc(\pi)=|\Asc(\pi)|=n-1-\des(\pi)$ the number of ascents of $\pi$.

Proposition \ref{prop: description reverse} gives a partial characterization of the periodic patterns of $\rs_k$.
For patterns of length $n \neq 2\bmod 4$, it states that $\pi\in\PP_n(\rs_k)$ if and only if $\hp$ can be written as a concatenation of $k$ decreasing sequences.
The next corollary follows from this description.
The case $n = 2 \bmod 4$ will be discussed in Section \ref{sec:reverse2odd}.

\begin{corollary}\label{cor: rev shift}
Let $\pi\in\S_n$, where $n \neq 2 \bmod 4$. Then $\pi\in\PP(\rs_k)$ if and only if $\asc(\hp)\le k-1$.
\end{corollary}

To enumerate periodic patterns of $\rs_k$ of length $n \neq  2 \bmod 4$, we use an argument very similar to the one we used for $\Sigma_k$.
For $1\le i\le n$, let $$C'(n,i)=|\{\tau\in\C_n:\asc(\tau)=i-1\}|.$$ By definition, we have $C'(n,i) = C(n,n-i+1)$.
The following result is analogous to Theorem~\ref{thm:enum2shift} for the reverse binary shift.

\begin{theorem}\label{thm:enum2rs}
For $n\geq 3$ with $n \neq 2 \bmod 4$, $$p_n(\rs_2) = C'(n,2) = L_2(n).$$
\end{theorem}

\begin{proof}
Since $n\geq 3$, there are no permutations in $\C_n$ with no ascents. By Theorem~\ref{thm: description of periodic patterns}, the map $\theta$ gives a bijection between $\PPP(\rs_2)$ and $\C_n^{--}=\{\tau\in\C_n:\asc(\tau)=1\}$, so $p_n(\rs_2)=C'(n,2)$.

Next we show that the map $\PPat_{\rs_2}:\W_{2,n}\to\PP_n(\rs_2)$ is a bijection, and therefore so is the map
$\PPPat_{\rs_2}:\WW_{2,n}\to\PPP_n(\rs_2)$, which implies that $p_n(\rs_2)=L_2(n)$ by Lemma~\ref{lem:periodic orbits}.
As in the proof of Theorem~\ref{thm:enum2shift}, it is enough to show that $\PPat_{\rs_2}$ is injective.
Let $\pi\in\PP_n(\rs_2)$, and let $s=(s_1\dots s_n)^\infty\in\W_{2,n}$ be such that $\PPat_{\rs_2}(s) = \pi$.
By Lemma~\ref{lem:iff}, $s_1\dots s_n$ is the $\pi$-monotone word induced by some $--$-segmentation of $\hp$. Since $\hp$ has one ascent, there is only one such segmentation, so $s$ is uniquely determined.
\end{proof}

\begin{theorem}\label{thm:enumrs}
For $n\geq 3$ with $n \neq 2 \bmod 4$ and $k\ge3$, $$p_n(\rs_k)-p_n(\rs_{k-1})= C'(n,k) = L_k(n) - \sum_{i = 2}^{k-1} \binom{n+k-i}{k-i}C'(n,i).$$
\end{theorem}

\begin{proof}
This proof is analogous to that of Theorem~\ref{thm: enum shift}.
By Corollary~\ref{cor: rev shift}, $\pi\in\PP_n(\rs_k)\setminus\PP_n(\rs_{k-1})$ if and only if $\asc(\hp)=k-1$, and so $\theta$ is a bijection between $\PPP_n(\rs_k)\setminus\PPP_n(\rs_{k-1})$
and cyclic permutations with $k-1$ ascents, from where $C'(n,k)=p_n(\rs_k)-p_n(\rs_{k-1})$.

For the second equality of the statement,  we find the cardinality of $\WW_{k,n}$ in two ways. By Lemma~\ref{lem:periodic orbits}, this number equals $L_k(n)$.
On the other hand, since the map $\PPPat_{\rs_k}:\WW_{k,n}\to\PPP_n(\rs_k)$ is surjective, but in general not injective, we can obtain $|\WW_{k,n}|$ by adding the cardinalities of the preimages of the elements of $\PPP_n(\rs_k)$ under this map.

For fixed $\pi\in\PP_n(\rs_k)$, we count how many words $s\in\W_{k,n}$ satisfy $\PPat_{\rs_k}(s)=\pi$ (equivalently, how many orbits in $\WW_{k,n}$ are mapped to $[\pi]$ by $\PPPat_{\rs_k}$).
By Lemma~\ref{lem:iff}, this number is equal to the number of $-^k$-segmentations of $\hp$.
If $\asc(\hp)=i-1$, there are $\binom{n+k-i}{k-i}$ such segmentations $0=e_0\le e_1\le \dots \le e_k=n$, since $\Asc(\hp)$ has to be a subset of $\{e_1,\dots,e_{k-1}\}$.

By Corollary~\ref{cor: rev shift}, for each $2\le i\le k$, the number of equivalence classes $[\pi]\in\PPP_n(\rs_k)$ where $\asc(\hp)=i-1$ is $C'(n,i)$.
It follows that $$L_k(n)=\sum_{i=2}^k \binom{n+k-i}{k-i} C'(n,i),$$
which is equivalent to the stated formula.
\end{proof}

Combining Theorems \ref{thm:enum2shift}, \ref{thm: enum shift}, \ref{thm:enum2rs} and~\ref{thm:enumrs}, we obtain the following.

\begin{corollary}\label{cor:CC'}
For $n \neq 2 \bmod 4$ and $2\le k\le n$, $$C(n,k) = C'(n,k).$$
\end{corollary}

This equality is equivalent to the symmetry $C(n,k)=C(n,n-1-k)$, which is not obvious from the recursive formula in Theorem \ref{thm: enum shift}.
Corollary~\ref{cor:CC'} also follows from a more general result of Gessel and Reutenauer \cite[Theorem 4.1]{GesReut},
which states that if $n \neq 2 \bmod 4$, then for any $D\subseteq[n-1]$,
\begin{equation}\label{eq:GR}|\{\tau\in\C_n:\Des(\tau)=D\}|= |\{\tau\in\C_n:\Asc(\tau)=D\}|.\end{equation}
The proof in \cite{GesReut} involves quasisymmetric functions. Even though we do not know of a direct bijection proving Eq.~\eqref{eq:GR}, our construction can be used to give the following bijection between
$\{\tau\in\C_n:\Des(\tau)\subseteq D\}$ and $\{\tau\in\C_n:\Asc(\tau)\subseteq D\}$.

Given $\tau\in\C_n$ such that $\Des(\tau)\subseteq D= \{d_1, d_2,\dots , d_{k-1}\}$ (where $d_1<\dots<d_{k-1}$), let $\pi\in\S_n$ be such that $\hp=\tau$.
Let $s=(s_1\dots s_n)^\infty\in\W_{k,n}$ be defined by $s_i=t$ if $d_t< \pi_i\leq d_{t+1}$, for $1\le i\le n$, where we let $d_0 = 0$ and $d_k = n$ (in our terminology, $s_1s_2\dots s_n$ is the $\pi$-monotone word induced by the $+^k$-segmentation $0,d_1,\dots,d_{k-1},n$ of $\hp$).
Let $\pi'=\PPat_{\rs_k}(s)$, and define $\delta(\tau)=\widehat{\pi'}=\theta([\pi'])$.

\begin{proposition}
When $n\neq2\bmod4$, the map $\delta$ defined above is a bijection between $\{\tau\in\C_n:\Des(\tau)\subseteq D\}$ and $\{\tau\in\C_n:\Asc(\tau)\subseteq D\}$.
\end{proposition}

\begin{proof}
Let $D= \{d_1, d_2,\dots , d_{k-1}\}$, where $0<d_1<\dots<d_{k-1}<n$, and let $\W_{k,n}^D$ be the set of words $w=(w_1\dots w_n)^\infty\in \W_{k,n}$ satisfying $|\{i\in[n]:w_i<t\}|=d_t$ for all $t$.

First note that the map $\delta$ is well defined, in the sense that $\delta(\tau)$ does not depend on the choice of $\pi\in\S_n$ with $\hp=\tau$. This is because any other choice of an element in $[\pi]$ would
produce a word in the same periodic orbit as $s$, and so its image under $\PPat_{\rs_k}$ would still be a permutation in $[\pi']$. The map $\delta$
can thus be viewed as a composition $\delta = \theta \circ \delta' \circ \theta^{-1}$, where $\delta'$ maps $[\pi]$ to $[\pi']$, with $\pi'$ as defined above.
To conclude that $\delta$ is a bijection, we will show that the map $\pi\mapsto\pi'$ is a bijection between $\PP_n^D(\Sigma_k):=\{\pi\in\PP_n(\Sigma_k):\Des(\hp)\subseteq D\}$ and $\PP_n^D(\rs_k):=\{\pi'\in\PP_n(\rs_k):\Asc(\widehat{\pi'})\subseteq D\}$.

Let us first show that the map $\PPat_{\Sigma_k}:\W_{k,n}\to\PP_n(\Sigma_k)$ restricts to a bijection between $\W_{k,n}^D$ and $\PP_n^D(\Sigma_k)$.
Given $w\in\W_{k,n}^D$, it is clear that $\PPat_{\Sigma_k}(w)\in\PP_n^D(\Sigma_k)$ by part 3 of Lemma~\ref{lem: ppat}. To see that the map is surjective, suppose that $\pi\in\PP_n^D(\Sigma_k)$,
and let $s$ be as defined above (that is, $s_1s_2\dots s_n$ is the $\pi$-monotone word induced by the $+^k$-segmentation $0,d_1,\dots,d_{k-1},n$ of $\hp$). Then $s\in\W_{k,n}^D$ and $\PPat_{\Sigma_k}(s)=\pi$ by Lemma~\ref{lem:iff}. Besides, since this is the only segmentation of $\hp$ whose induced $\pi$-monotone word is in $\W_{k,n}^D$, Lemma~\ref{lem:iff} implies that $s$ is unique, and so the map is injective.

Similarly, the map $\PPat_{\rs_k}:\W_{k,n}\to\PP_n(\rs_k)$ restricts to a bijection between $\W_{k,n}^D$ and $\PP_n^D(\rs_k)$.
Given $w\in\W_{k,n}^D$, part 3 of Lemma~\ref{lem: ppat} shows that $\PPat_{\rs_k}(w)\in\PP_n^D(\rs_k)$. To see that the map is surjective, suppose that $\pi'\in\PP_n^D(\rs_k)$, and let
$s=(s_1s_2\dots s_n)^\infty$, where $s_1s_2\dots s_n$ is the $\pi'$-monotone word induced by the $-^k$-segmentation $0,d_1,\dots,d_{k-1},n$ of $\widehat{\pi'}$. Then $s\in\W_{k,n}^D$ and $\PPat_{rs_k}(s)=\pi'$ by Lemma~\ref{lem:iff}. Since this is the only segmentation of $\widehat{\pi'}$ whose induced $\pi'$-monotone word is in $\W_{k,n}^D$, the word $s$ is unique, and so the map is injective.

Finally, the map $\pi\mapsto\pi'$ is a bijection because it is the composition of the above two bijections, as shown in this diagram:
$$\begin{array}{ccccc}
\PP_n^D(\Sigma_k)&\overset{\PPat_{\Sigma_k}}{\longleftarrow}& \W_{k,n}^D&\overset{\PPat_{\rs_k}}{\longrightarrow}&\PP_n^D(\rs_k) \\
\pi&\mapsto&s&\mapsto&\pi'.
\end{array}$$
\end{proof}

Let us see an example of the bijection $\delta$ for $k = 3$ and $n = 9$. Let $D = \{3, 7\}$, and suppose that $\tau = 245378916$. Then we get $$\pi = 124357968 \mapsto s = (001011212)^\infty \mapsto \pi' = 317265849 \mapsto \delta(\tau)=\widehat{\pi'} = 761985243.$$ In this case $\Des(\tau) = \Asc(\delta(\tau))$, but this is not the case in general.

Another consequence of Theorem~\ref{thm:enumrs} is that, when $n \neq2\bmod 4$, $$p_n(\rs_k)=\sum_{i = 2}^k C'(n,k)=p_n(\Sigma_k).$$

\subsection{The reverse $k$-shift, when $n = 2 \mbox{ mod } 4$}\label{sec:reverse2odd} 

The results in Section~\ref{sec:reverse} do not apply to the case $n = 2 \bmod 4$.
Corollary \ref{cor: rev shift} fails in that there are certain permutations $\pi\in\S_n$ with $\asc(\hp)\le k-1$ that are not periodic patterns of $\rs_k$.
In other words, the map $\theta:\PPP_n(\rs_k)\to\C_n^{-^k}$ is not surjective. For $k=2$, the following result is the analogue of Theorem~\ref{thm:enum2rs}.

\begin{theorem} \label{thm:enum2rs2mod4}
For $n\geq 3$ with $n = 2\bmod 4$,
$$p_n(\rs_2) = C'(n, 2) -C'(\frac{n}{2},2)=L_2(n).$$
\end{theorem}

\begin{proof}
To prove that $p_n(\rs_2) = L_2(n)$, we show that the map $\PPat_{\rs_2}:\W_{2,n}\to\PP_n(\rs_2)$ is a bijection, and therefore so is the map $\PPPat_{\rs_2}:\WW_{2,n}\to\PPP_n(\rs_2)$.
As in the proof of Theorem~\ref{thm:enum2shift}, it is enough to show that $\PPat_{\rs_2}$ is injective.
Let $\pi\in\PP_n(\rs_2)$, and note that $\asc(\hp)=1$ by Proposition~\ref{prop: description reverse}, since there are no cycles with no descents of length $n\ge3$.
Suppose that $s=(s_1\dots s_n)^\infty\in\W_{2,n}$ is such that  $\PPat_{\rs_2}(s) = \pi$.
If $d$ is the number of zeros in $s_1\dots s_n$, we have by part 3 of Lemma \ref{lem: ppat} that
$0\le d\le n$ is a $--$-segmentation of $\hp$, and so $\Asc(\hp)=\{d\}$.
It follows by part 1 of Lemma \ref{lem: ppat} the word $s_1\dots s_n$ is uniquely determined.

Next we prove that $p_n(\rs_2) = C'(n, 2) -C'(\frac{n}{2},2).$
Let $r=n/2$.
By Proposition~\ref{prop: description reverse}, the map $\theta$ restricts to an injection from $\PPP_n(\rs_2)$ to $\C_n^{--}=\{\tau\in\C_n:\asc(\tau)=1\}$.
It suffices to show that the cardinality of $\C_n^{--}\setminus \theta(\PPP_n(\rs_2))$, or equivalently, the cardinality of
$B=\theta^{-1}(\C_n^{--})\setminus\PPP_n(\rs_2)$, is $C'(r,2)$.
We do this by constructing a bijection $\overline{\varphi}$ between $B$ and $\PPP_r(\rs_2)$. Using that $\PPP_r(\rs_2)=\{[\rho]\in\SS_r:\asc(\hat\rho)=1\}$ by Proposition~\ref{prop: description reverse},
since $r$ is an odd number, the map $\theta$ gives a bijection from $\PPP_r(\rs_2)$ to $\C_r^{--}$, and so the sequence of bijections in the second row of Figure~\ref{fig:varphi2} proves that
$|B|=p_r(\rs_2)=C'(r,2)$.

\begin{figure}[htb]
\centering
\begin{tikzpicture}
  \matrix (m) [matrix of math nodes,row sep=2.5em,column sep=2em,minimum width=2em]
  {
     & \{\pi \in \S_n : \asc(\pihat) = 1\}\setminus \PP_n(\rs_2) & \PP_r(\rs_2) &  \\
     \C_n^{--}\setminus\theta(\PPP_n(\rs_2)) & B = \{[\pi]\in\SS_n: \asc(\pihat) = 1 \}\setminus \PPP_n(\rs_2) & \PPP_r(\rs_2) & \C_r^{--} \\
     };
  \path[-stealth]
    (m-1-2) edge [dashed] (m-2-2)
            edge node [above] {\footnotesize $\varphi$} (m-1-3)
    (m-1-3) edge [dashed] (m-2-3)
            edge node [above right] {$\rho\mapsto\hat\rho$} (m-2-4)
    (m-2-2) edge node [above] {\footnotesize $\overline{\varphi}$} node [below] {$\sim$} (m-2-3)
    (m-2-3) edge node [above] {$\theta$} node [below] {$\sim$} (m-2-4)
    (m-2-2) edge node [above] {$\theta$} node [below] {$\sim$} (m-2-1);
\end{tikzpicture}
\caption{The maps used in the proof of Theorem~\ref{thm:enum2rs2mod4}. Dashed arrows denote maps sending a permutation to its equivalence class of cyclic rotations. The symbols $\sim$ indicate bijections.}
\label{fig:varphi2}
\end{figure}

Let us first define a map $\varphi$ from $\{\pi\in\S_n:\asc(\hp)=1\}\setminus\PP_n(\rs_2)$ to $\PP_r(\rs_2)$ as follows. Given $\pi$ in the former set, note that the fact that $\asc(\hp)=1$ implies that $\hp$ has a unique $--$-segmentation.
Let $s_1s_2\dots s_n$ be the $\pi$-monotone binary word induced by this segmentation. This word cannot be primitive, because if it were, then Lemma~\ref{lem:sisj} would imply that
$\PPat_{\rs_2}((s_1\dots s_n)^\infty)=\pi$, contradicting that $\pi\notin\PP_n(\rs_2)$. In fact, the proof of Lemma~\ref{lem: s1sn} shows that the only way for this word not to be primitive is if
$s_1\dots s_n=q^2$ for some primitive binary word $q$ of length $r$. Let $\rho=\PPat_{\rs_2}(q^\infty)$, and define $\varphi(\pi)=\rho$.  It is clear from the above construction that cyclic rotations of $\pi$ produce cyclic rotations of the word $s_1s_2\dots s_n$, which in turn lead to cyclic rotations of $q$ and $\rho$.
If $\varphi(\pi)=\rho$, then the $n$ elements of $[\pi]$ are mapped by $\varphi$ to the $r$ elements of $[\rho]$,
so $\varphi$ induces a map $\overline{\varphi}$ from $B=\{[\pi]\in\SS_n:\asc(\hp)=1\}\setminus\PPP_n(\rs_2)$ to $\PPP_r(\rs_2)$.
We will prove that $\varphi$ is a 2-to-1 map, which is equivalent to the fact that $\overline{\varphi}$ is a bijection.

First we prove that $\varphi$ is surjective. Let $\rho\in\PP_r(\rs_2)$. For convenience, define $\rho_{i+r}=\rho_i$ for $1\le i\le r$. Define a permutation $\pi\in\S_n$ by
\begin{equation}\label{eq:defpi}\pi_i = \begin{cases}2\rho_i-1 & \mbox {if $i$ is odd,}\\  2\rho_{i} & \mbox{if $i$ is even.}\end{cases}\end{equation}
Note that this construction produces a permutation because $r$ is odd. We will show that $\asc(\hp)=1$, $\pi\notin\PP_n(\rs_2)$, and $\varphi(\pi)=\rho$.

We claim that for $1\le i\le r$, we have $\hp_{2i-1}=2\hat\rho_i$ and $\hp_{2i}=2\hat\rho_i-1$.
Indeed, letting $1\le j\le r$ be such that $\rho_j=i$, it follows from the definition of $\pi$ that $\hp_{2i-1}=\hp_{2\rho_j-1}=2\rho_{j+1}=2\hat\rho_{\rho_j}=2\hat\rho_i$,
and similarly  $\hp_{2i}=\hp_{2\rho_j}=2\rho_{j+1}-1=2\hat\rho_{\rho_j}-1=2\hat\rho_i-1$.
It is clear from this description of $\hp$ that $\asc(\hp)=\asc(\hat\rho)=1$. Moreover, if $\Asc(\hat\rho)=\{d\}$, then $\Asc(\hp)=\{2d\}$, which determines the unique $--$-segmentation of $\hp$.

Now suppose that there is some $w\in\W_{2,n}$ such that $\pi = \PPat_{\rs_2}(w)$. By part 1 of Lemma \ref{lem: ppat},
$w_1w_2\dots w_n$ is a $\pi$-monotone word having $2d$ zeros. Since by construction $\{\pi_i,\pi_{i+r}\}=\{2\rho_i,2\rho_i-1\}$  for all $1\le i\le r$,
a binary $\pi$-monotone word with an even number of zeros satisfies $w_1w_2\dots w_n=(w_1w_2\dots w_r)^2$, and so it is not primitive, which contradicts that $w\in\W_{2,n}$.
It follows that $\pi\notin\PP_n(\rs_2)$.

To show that $\varphi(\pi)=\rho$, let $s_1s_2\dots s_n$ be the $\pi$-monotone word induced by the unique $--$-segmentation $0\le 2d\le n$ of $\hp$. By definition, $\varphi(\pi)=\PPat_{\rs_2}((s_1\dots s_r)^\infty)$.
On the other hand, $s_1s_2\dots s_r$ is a $\rho$-monotone word with $d$ zeros, and thus induced by the
$--$-segmentation $0\le d\le r$ of $\hat\rho$. Lemma~\ref{lem:iff} implies that $\PPat_{\rs_2}((s_1\dots s_r)^\infty)=\rho$.

It remains to show that every $\rho\in\PP_r(\rs_2)$ has exactly two preimages under~$\varphi$. Let $\pi$ be such that $\varphi(\pi)=\rho$. By construction of $\varphi$, we have $\varphi(\pi_{r+1}\dots\pi_n\pi_1\dots\pi_r)=\rho$ as well.
We will show that there are no other elements in $\{\pi\in\S_n:\asc(\hp)=1\}\setminus\PP_n(\rs_2)$ whose image by $\varphi$ is $\rho$. Let $s\in\W_{2,r}$ be such that $\PPat_{\rs_2}(s)=\rho$, and note that $s$ is unique because $\PPat_{\rs_2}$ is a bijection. Since $\varphi(\pi)=\rho$, we know that $s_1s_2\dots s_n=(s_1s_2\dots s_r)^2$ is the $\pi$-monotone word induced by the unique $--$-segmentation of $\hp$.
Letting $d$ be the number of zeros in $s_1\dots s_r$, this segmentation is $0\le 2d\le n$.
If $1\le i, j\le n$ are such that $s_{[i,\infty)} \neq s_{[j, \infty)}$, then Lemma \ref{lem:sisj} implies that $\pi_i<\pi_j$ if and only if $s_{[i,\infty)}\prec_{\sigma} s_{[j,\infty)}$. It follows that $\rho = \st(\pi_1\dots \pi_r)$ and $\rho = \st(\pi_{r+1}\dots \pi_n)$.
In addition, since $s_{[i,\infty)} = s_{[i+r, \infty)}$ and $s_{[j,\infty)} = s_{[j+r, \infty)}$ for $1\le i,j\le r$, the four inequalities $s_{[i,\infty)}\prec_{\sigma} s_{[j,\infty)}$, $s_{[i+r,\infty)}\prec_{\sigma} s_{[j+r,\infty)}$, $s_{[i+r,\infty)}\prec_{\sigma} s_{[j,\infty)}$ and $s_{[i,\infty)}\prec_{\sigma} s_{[j+r,\infty)}$ are equivalent, and so the four inequalities  $\pi_i<\pi_j$, $\pi_{i+r}<\pi_{j+r}$, $\pi_i<\pi_{j+r}$, $\pi_{i+r}<\pi_j$
must be either all true or all false.
This implies that $\{\pi_i, \pi_{i+r}\}=\{2\rho_i-1, 2\rho_i\}$ for $1\le i\le r$. We claim that once the relative order of $\pi_1$ and $\pi_{1+r}$ is chosen, the rest of the entries of $\pi$ are uniquely determined by induction. Indeed, having determined the relative order of $\pi_i$ and $\pi_{i+r}$, say $2\rho_i-1=\pi_i<\pi_{i+r}=2\rho_i$, then Lemma~\ref{lem: Csigma} implies that $\pi_{i+1}>\pi_{i+r+1}$, so these entries are forced to be $\pi_{i+1}=2\rho_{i+1}$ and $\pi_{i+r+1}=2\rho_{i+1}-1$. Note that Lemma~\ref{lem: Csigma} can be applied because either $\pi_i,\pi_{i+r}\le 2d$ or $\pi_i,\pi_{i+r}>2d$.
\end{proof}

Here is an example to illustrate the above proof of the fact that $\varphi$ is surjective. Suppose that
$\rho=14523\in\PP_5(\rs_2)$.
Then Equation~\eqref{eq:defpi} yields $\pi = 1\,8\,9\,4\,5\,2\,7\,10\,3\,6$. In this case, $\hat{\rho} = 43152$ and $\hp = 8\,7\,6\,5\,2\,1\,10\,9\,4\,3$. It is easy to check that indeed $\varphi(\pi)=\rho$, since the $\pi$-monotone word induced by the $--$-segmentation $0\le 6\le 10$ of $\hp$ is $0110001100$, and $\PPat_{\rs_2}((01100)^\infty)=14523$.

Combining Theorems~\ref{thm:enum2shift}, \ref{thm:enum2rs} and \ref{thm:enum2rs2mod4}, we get the following relationship between cycles with one ascent and cycles with one descent.

\begin{corollary} \label{cor: k=2 relation} For $n\ge3$,
$$C'(n,2) = \begin{cases} C(n, 2) + C(\frac{n}{2},2) & \mbox{if } n = 2 \bmod 4, \\ C(n,2) & \mbox{otherwise.}\end{cases}$$
\end{corollary}

Next we extended the argument used to prove Theorem~\ref{thm:enum2rs2mod4} to arbitrary $k$, obtaining a formula for the number of periodic patterns of the reverse $k$-shift in terms of the number of cyclic permutations with a fixed number of ascents.

\begin{theorem}\label{thm:enumrs2mod4}
For $n\geq 3$ with $n = 2\bmod 4$ and $k\ge3$,
\begin{equation}\label{eq:enumrs}p_n(\rs_k) = \sum_{i=2}^k C'(n,i) - C'(\frac{n}{2}, k).\end{equation}
\end{theorem}

\begin{proof}
By Proposition~\ref{prop: description reverse}, the map $\theta$ restricts to an injection from $\PPP_n(\rs_k)$ to $\C_n^{-^k}=\{\tau\in\C_n:\asc(\tau)\leq k-1\}$. Noting that $|\C_n^{-^k}| = \sum_{i=2}^k C'(n,i)$, it follows that
$$
p_n(\rs_k) =  \sum_{i=2}^{k} C'(n,i) - |\C_n^{-^k}\setminus\theta(\PPP_n(\rs_k))|.
$$
Let $r = n/2$. It suffices to show that $|\C_n^{-^k}\setminus\theta(\PPP_n(\rs_k))|$, or equivalently, the cardinality of $B = \theta^{-1}(\C_n^{-^k})\setminus \PPP_n(\rs_k)$, is $C'(r, k)$. We do this by constructing a bijection $\overline{\varphi}$ from $B$ to $\{[\rho] \in \SS_r : \asc(\hat\rho) = k-1\}$. Since $\theta$ clearly gives a bijection between the latter set and the set of cycles in $\C_r$ with $k-1$ ascents, it will follow
that $|B|=C'(r, k)$. A diagram of the maps used in this proof appears in Figure~\ref{fig:varphi}. It is useful to note that $\theta^{-1}(\C_n^{-^k})=\{[\pi]\in\SS_n: \asc(\pihat) \leq k-1\}$.

\begin{figure}[htb]
\centering
\begin{tikzpicture}
  \matrix (m) [matrix of math nodes,row sep=1.4em,column sep=1em,minimum width=2em]
  {
     \{\pi \in \S_n : \asc(\pihat) \leq k{-}1\}\setminus \PP_n(\rs_k) & \{\rho \in \S_r : \asc(\hat\rho) = k{-}1\} &  \\  && \\
     B {=} \{[\pi]\in\SS_n: \asc(\pihat) \leq k{-}1\}\setminus \PPP_n(\rs_k) & \{[\rho] \in \SS_r : \asc(\hat\rho) = k{-}1\} & \{\hat\rho\in\C_r : \asc(\hat\rho)=k{-}1\} \\
      \C_n^{-^k}\setminus\theta(\PPP_n(\rs_k)) &  &  \\
     };
  \path[-stealth]
    (m-1-1) edge [dashed] (m-3-1)
            edge node [above] {\footnotesize $\varphi$} (m-1-2)
    (m-1-2) edge [dashed] (m-3-2)
            edge node [above right] {$\rho\mapsto\hat\rho$} (m-3-3)
    (m-3-1) edge node [above] {\footnotesize $\overline{\varphi}$} node [below] {$\sim$} (m-3-2)
    (m-3-2) edge node [above] {$\theta$} node [below] {$\sim$} (m-3-3)
    (m-3-1) edge node [left] {$\theta$} node [right] {$\wr$} (m-4-1);
\end{tikzpicture}
\caption{The maps used in the proof of Theorem~\ref{thm:enumrs2mod4}.}
\label{fig:varphi}
\end{figure}

We will first define a map $\varphi$ from $\{\pi \in \S_n : \asc(\pihat) \leq k-1\}\setminus \PP_n(\rs_k)$ to $\{\rho \in \S_r : \asc(\hat\rho) = k-1\}$.
We claim that if $\pi$ is in the former set, then $\asc(\pihat) = k-1$.
Indeed, for every $-^k$-segmentation $0=e_0\leq \cdots \leq e_k = n$ of $\hp$, the $\pi$-monotone word $s_1\cdots s_n$ induced by it is not primitive,
because otherwise, by Lemma~\ref{lem:sisj}, it would satisfy
$\PPat_{\rs_k}((s_1\dots s_n)^\infty)=\pi$, contradicting that $\pi\notin\PP_n(\rs_k)$. In fact, this word has to be of the form $s_1\cdots s_n = q^2$, and so all the $e_i$ are even, since each letter appears an even number of times.
But if $\asc(\hp) < k-1$, then $\hp$ would have $-^k$-segmentations without this property, since any sequence $0=e_0\leq \cdots \leq e_k = n$ such that $\Asc(\hp)\subseteq\{e_1,\dots,e_{k-1}\}$ would be a $-^k$-segmentation of $\hp$.

To define the map $\varphi$, let $\pi \in \S_n\setminus \PP_n(\rs_k)$ be such that $\asc(\hp) = k-1$, and let
$s_1\cdots s_n$ be the $\pi$-monotone word induced by the unique $-^k$-segmentation of $\hp$.
By the proof of Lemma~\ref{lem: s1sn}, we have $s_1\cdots s_n = q^2$ for some primitive word $q$ of length $r$.
Let $\rho = \PPat_{\rs_k}(q^\infty)$, and define $\varphi(\pi) = \rho$.

Let us first check that such $\rho$ satisfies $\asc(\hat\rho) = k-1$. Let $s = (s_1\cdots s_n)^\infty = q^\infty$. Denote the $k-1$ ascents of $\hp$ by $e_1,e_2,\dots,e_{k-1}$, in increasing order, so that $0 = e_0< e_1<\cdots<e_k = n$ is the unique $-^k$-segmentation of $\hp$. Note that $e_t$ must be even for all $t$, since each letter appears an even number of times in $q^2$. Letting $e_t=2d_t$,
we will show that $d_t \in \Asc(\hat\rho)$ for $1\le t\le k-1$. Since $k-1$ is an upper bound on $\asc(\hat\rho)$ by Proposition~\ref{prop: description reverse}, it will follow that $\asc(\hat\rho) = k-1$.
Let $i,j$ be such that $\pi_i = e_t$ and $\pi_j = e_t+1$. By construction of $s_1\dots s_n$, we have $s_j=s_i+1$, and so $i \neq j \bmod r$.
Taking indices $\bmod\, n$, we have $\pi_{i+1}=\hp_{\pi_i}=\hp_{e_t}<\hp_{e_t+1}=\hp_{\hp_j}=\pi_{j+1}$. Thus, $s_{[i+1,\infty)}\prec_\sigma s_{[j+1,\infty)}$, because by Lemma~\ref{lem:sisj}, for any $1\le a,b\le n$ such that $a\neq b\bmod r$ and $\pi_a<\pi_b$, we must have $s_{[a,\infty)}\prec_\sigma s_{[b,\infty)}$.
 By construction, the word $s_1s_2\dots s_n=q^2$ has $e_t=2d_t$ entries less than or equal to $s_i$, namely those entries $s_\ell$ for which $\pi_\ell\le\pi_i$. By Lemma~\ref{lem:sisj}, these entries satisfy $s_{[\ell,\infty)}\preceq_\sigma s_{[i,\infty)}$, with equality only when $\ell=i\bmod r$. Thus, letting $i'=i\bmod r$, the word $s_1s_2\dots s_r=q$ has $d_t$ entries $s_\ell$ such that $s_{[\ell,\infty)}\preceq_\sigma s_{[i',\infty)}$, and so $\rho_{i'}=d_t$.
Similarly, $s_1s_2\dots s_n$ has $n-2d_t$ entries $s_\ell$ satisfying $s_{[j,\infty)}\preceq_\sigma s_{[\ell,\infty)}$, with equality only when $\ell=j\bmod r$. Thus, letting $j'=j\bmod r$, the word $q$ has $r-d_t$ entries $s_\ell$ such that $s_{[j',\infty)}\preceq_\sigma s_{[\ell,\infty)}$, and so $\rho_{j'}=d_t+1$.
Since $s_{[i+1,\infty)}\prec_\sigma s_{[j+1,\infty)}$, we have $\rho_{i+1} < \rho_{j+1}$, taking indices $\bmod\, r$. It follows that $\hat\rho_{d_t}=\hat\rho_{\rho_i}=\rho_{i+1}<\rho_{j+1}=\hat\rho_{\rho_j}=\hat\rho_{d_t+1}$, so $d_t \in \Asc(\hat\rho)$ as claimed.

It is clear from the definition of $\varphi$ that cyclic rotations of $\pi$ lead to cyclic rotations of $s_1\cdots s_n$, which in turn lead to cyclic rotations of $\rho$.
If $\varphi(\pi)=\rho$, then the $n$ elements of $[\pi]$ are mapped by $\varphi$ to the $r$ elements of $[\rho]$, so $\varphi$ induces a map $\overline{\varphi}$ from $B =
\{[\pi]\in \SS_n: \asc(\pihat) = k-1\}\setminus \PPP_n(\rs_k)$ to $\{[\rho] \in \S_r : \asc(\hat\rho) = k-1\}$. We will prove that $\varphi$ is a 2-to-1 map, which is equivalent to the fact that $\overline{\varphi}$ is a bijection.

First, we prove that $\varphi$ is surjective. Let $\rho\in \S^r$ be such that $\asc(\hat{\rho}) = k-1$.
For convenience, define $\rho_{i+r} = \rho_i$ for $1\leq i \leq r$. Define a permutation $\pi \in \S_n$ by
$$\pi_i = \begin{cases}2\rho_i-1 & \mbox {if $i$ is odd,}\\  2\rho_{i} & \mbox{if $i$ is even.}\end{cases}$$
We will show that $\asc(\hp)=k-1$, $\pi\notin\PP_n(\rs_k)$, and $\varphi(\pi)=\rho$.

As in the proof of Theorem~\ref{thm:enum2rs2mod4}, we have $\hp_{2i-1}=2\hat\rho_i$ and $\hp_{2i}=2\hat\rho_i-1$ for $1\le i\le r$.
It follows that $\asc(\hp)=\asc(\hat\rho)=k-1$. Moreover, if $\Asc(\hat\rho)=\{d_1,d_2,\cdots, d_{k-1}\}$, then $\Asc(\hp)=\{2d_1,2d_2,\cdots, 2d_{k-1}\}$, which determines the unique $-^k$-segmentation of $\hp$.

Suppose that there is some word $w \in \W_{k,n}$ such that $\pi = \PPat_{\rs_k}(w)$. By part 1 of Lemma \ref{lem: ppat},
$w_1w_2\dots w_n$ is a $\pi$-monotone word having $2d_{t+1} - 2d_t$ copies of the letter $t$ for $0\le t\le k-1$ (with the convention $d_0=0$ and $d_k=r$).
Since by construction $\{\pi_i,\pi_{i+r}\}=\{2\rho_i,2\rho_i-1\}$  for all $1\le i\le r$,
a $\pi$-monotone word with an even number of copies of each letter satisfies $w_1w_2\dots w_n=(w_1w_2\dots w_r)^2$, and so it is not primitive, which contradicts that $w\in\W_{k,n}$.
It follows that $\pi\notin\PP_n(\rs_k)$.

To show that $\varphi(\pi)=\rho$, let $s_1s_2\dots s_n$ be the $\pi$-monotone word induced by the unique $-^k$-segmentation $0\le 2d_1\le 2d_2\le \cdots\le 2d_{k-1}\le n$ of $\hp$.
By definition, $\varphi(\pi)=\PPat_{\rs_k}((s_1\dots s_r)^\infty)$.
On the other hand, $s_1s_2\dots s_r$ is a $\rho$-monotone word
having $d_{t+1}- d_t$ copies of $t$ for each $0\leq t \leq k-1$, and thus induced by the
$--$-segmentation $0\le d_1\le d_2\le \cdots\le d_{k-1}\le r$ of $\hat\rho$. Lemma~\ref{lem:iff} implies that $\PPat_{\rs_k}((s_1\dots s_r)^\infty)=\rho$.

It remains to show that every $\rho \in \S_r$ with $\asc(\hat\rho) = k-1$ has exactly two preimages under $\varphi$. Suppose that $\pi$ is such that $\varphi(\pi)=\rho$. By construction of $\varphi$, we have $\varphi(\pi_{r+1}\dots\pi_n\pi_1\dots\pi_r)=\rho$ as well.
We will show that there are no other elements in $\{\pi\in\S_n:\asc(\hp)= k-1\}\setminus\PP_n(\rs_k)$ whose image by $\varphi$ is $\rho$.

Since $r$ is odd and $\asc(\hat\rho) = k-1$, Proposition \ref{prop: description reverse} implies that $\rho\in\PP_r(\rs_k)$, so there is some $s\in\W_{k,r}$ such that $\rho = \PPat_{\rs_k}(s)$.
In fact, such an $s$ is unique by Lemma~\ref{lem:iff}, because it is the $\rho$-monotone word induced by
the unique $-^k$-segmentation of $\hat\rho$, which we denote by $0\le d_1\leq \cdots \le d_{k-1}\le n$.
Since $\varphi(\pi)=\rho$, we know that $s_1s_2\dots s_n=(s_1s_2\dots s_r)^2$ is the $\pi$-monotone word induced by the unique $-^k$-segmentation of $\hp$. It follows that this segmentation must be $0\le 2d_1\leq \cdots \le 2d_{k-1}\le n$.
Exactly as in the proof of Theorem~\ref{thm:enum2rs2mod4}, we can show that $\{\pi_i, \pi_{i+r}\}=\{2\rho_i-1, 2\rho_i\}$ for $1\le i\le r$, and that once the relative order of $\pi_1$ and $\pi_{1+r}$ is chosen, the rest of the entries of $\pi$ are uniquely determined by induction. Lemma~\ref{lem: Csigma} can be applied in this case because for every $1\le i\le r$, we have $2d_t< \pi_i,\pi_{i+r}\le 2d_{t+1}$ for some $0\leq t\leq k-1$, where $d_0=0$ and $d_k=r$.
\end{proof}

Theorem~\ref{thm:enumrs2mod4} alone does not give a practical way to compute $p_n(\rs_k)$ when $n=2\bmod4$, unlike Theorem~\ref{thm:enumrs} for the case $n\neq2\bmod4$.
Note that since $\frac{n}{2} \neq 2 \bmod 4$, the term
$C'(\frac{n}{2},k)$ on the right hand side of Equation~\eqref{eq:enumrs} can be easily computed by the recurrence in Theorem~\ref{thm:enumrs}.
Our next goal is to give a recurrence to compute $C'(n,i)$ when $n=2\bmod 4$.

\begin{theorem}
For $n\geq 3$ with $n = 2\bmod 4$ and $k\ge3$,
$$L_k(n) = \sum_{i=2}^k \left[\binom{n+k-i}{k-i}C'(n,i) - \binom{\frac{n}{2} + k -i}{k-i}C'(\frac{n}{2}, i)\right].$$
\end{theorem}

\begin{proof}
Since $L_k(n)=|\WW_{k,n}|$ by Lemma~\ref{lem:periodic orbits}, it is enough to show that the right hand side also counts $n$-periodic orbits of $\rs_k$.

Consider the map $\PPPat_{\rs_k}:\WW_{k,n}\to\PPP_n(\rs_k)\subseteq\{[\pi]\in\SS_n:\asc(\hp) \leq k-1\}$.
We can obtain $|\WW_{k,n}|$ by adding the cardinalities of the preimages of the elements of $\PPP_n(\rs_k)$ under this map, or equivalently, the cardinalities of the preimages of the elements of $\PP_n(\rs_k)$ under the map
$\PPat_{\rs_k}$, as in the proof of Theorems \ref{thm: enum shift} and \ref{thm:enumrs}.
However, the difference here is that not every permutation $\pi \in \S_n$ with $\asc(\hp) \leq k-1$
is in the image of this map, because when $n=2\bmod4$, Corollary~\ref{cor: rev shift} does not apply.

Let $\pi \in \S_n$ be such that $\asc(\hp)=i-1$, where $i\leq k$. If $s\in\W_{k,n}$ is such that  $\PPat_{\rs_k}(s)=\pi$, then, by
Lemma~\ref{lem: ppat}, $s_1s_2\dots s_n$ is the $\pi$-monotone word induced by some $-^k$-segmentation
$0=e_0\le e_1\le \dots \le e_k=n$ of $\hp$. There are $\binom{n+k-i}{k-i}$ such segmentations, since $\Asc(\hp)$ has to be a subset of $\{e_1,\dots,e_{k-1}\}$.
However, unlike in the proof of Theorem \ref{thm:enumrs}, it is not the case that the $\pi$-monotone word induced by every such segmentation is necessarily primitive.
When this word is not primitive, the proof of Lemma \ref{lem: s1sn} implies that it must be a {\it square}, i.e., a word of the form $q^2$ for some primitive word $q$ of length $r:=n/2$.
We will count how many times this induced word is a square.
First, note that if the $\pi$-monotone word induced by some $-^k$-segmentation of $\hp$
is a square, then the $\pi$-monotone word induced by the unique $-^i$-segmentation
of $\hp$ (which consists of the ascents of $\Asc(\hp)$, and thus is a subset of any $-^k$-segmentation)
must be a square as well. It follows that $\pi \notin \PP_n(\rs_i)$ in this case.
For each such $\pi$, any $-^k$-segmentation $0 = e_0\leq e_1\leq \dots \leq e_k = n$ of $\hp$ whose induced $\pi$-monotone word is a square must satisfy
$\Asc(\pihat) \subseteq \{e_1,\cdots, e_{k-1}\}$ (by definition of $-^k$-segmentation) and have all $e_t$ even. Conversely, every segmentation satisfying these conditions
induces a square $\pi$-monotone word, because, as shown in the proof of Theorem \ref{thm:enumrs2mod4}, for every $1\le j\le r$ the entries $\pi_j$ and $\pi_{j+r}$ have
consecutive values, with the largest one being even. The number of $-^k$-segmentations satisfying the two conditions is $\binom{r + k -i}{k-i}$.

For each $2\le i\le k$, there are $C(n,i)$ equivalence classes $[\pi]\in\SS_n$ with $\asc(\hp)=i-1$. We showed in the proof of Theorem \ref{thm:enumrs2mod4} that
$C'(r, i)$ of them satisfy that $\pi\notin  \PP_n(\rs_i)$; in other words, that $|\{[\pi]\in\SS_n: \asc(\pihat)=i-1\}\setminus \PPP_n(\rs_i)|=C'(r, i)$.
It follows that the number of preimages under  $\PPPat_{\rs_k}$ of the set $\{[\pi]\in\SS_n: \asc(\pihat)=i-1\}$ is
$$\binom{n+k-i}{k-i}C'(n,i) - \binom{r + k -i}{k-i}C'(r, i).$$
\end{proof}

\begin{corollary}\label{cor:recC'}
For $n\geq 3$ with $n = 2\bmod 4$ and $k\ge3$,
$$C'(n,k) = L_k(n) - \sum_{i=2}^{k-1} \left[\binom{n+k-i}{k-i}C'(n,i) - \binom{\frac{n}{2} + k -i}{k-i}C'(\frac{n}{2}, i)\right] + C'(\frac{n}{2}, k).$$
\end{corollary}

The equality $C'(n,k)=C(n,k)$, which holds for $n\neq2\bmod4$ (see Corollary~\ref{cor:CC'}),
is no longer valid when $n=2\bmod4$. In this case, we know that $C'(n,2)=C(n,2)+C(\frac{n}{2},2)$ by Corollary \ref{cor: k=2 relation}. For general $k$, an intricate
formula expressing $C'(n,k)$ in terms of $C(n,k)$ and $C(\frac{n}{2},i)$ for $i\le k$ can be derived from Corollary~\ref{cor:recC'}.

\section{Pattern-avoiding cyclic permutations} \label{sec:patternsincycles}

Using Theorem~\ref{thm: description of periodic patterns}, the formulas that we have found for the number of periodic patterns of the tent map, the $k$-shift and the reverse $k$-shift have implications to the enumeration of cyclic permutations
that avoid certain patterns.

We denote by $\C_n(\rho^{(1)},\rho^{(2)},\dots)=\C_n\cap\Av(\rho^{(1)},\rho^{(k)},\dots)$ the set of cycles of length $n$ avoiding the patterns $\rho^{(1)},\rho^{(2)},\dots$.
The enumeration of pattern-avoiding cycles is a wide-open problem, part of its difficulty stemming from the fact that it combines two different ways to look at permutations:
in terms of their cycle structure and in terms of their one-line notation. The question of finding a formula for $|\C_n(\sigma)|$ where $\sigma$ is a pattern of length~3 was proposed by
Richard Stanley~\cite{Sta}
and is still open. However, we are able to answer some related questions in Theorem~\ref{thm:patternsincycles}.
The first formula below, which counts unimodal cycles, was first obtained by Weiss and Rogers~\cite{WR} using methods from~\cite{MTh}.
More generally, the cycle structure of unimodal permutations has been studied by Gannon~\cite{Gan} and Thibon~\cite{Thi}.
The other formulas in Theorem~\ref{thm:patternsincycles} are new to the extent of our knowledge.

\begin{theorem}\label{thm:patternsincycles}
For $n\ge2$,
\begin{align*}
|\C_n(213,312)|=|\C_n(132,231)|&=\frac{1}{2n}\sum_{\substack{d|n \\ d \:  \mathrm{ odd} }} \mu(d) 2^{n/d},\\
|\C_n(321,2143,3142)|&=\frac{1}{n}\sum_{d|n} \mu(d) 2^{n/d},\\
|\C_n(123,2413,3412)|&=
\begin{cases} \displaystyle \frac{1}{n}\sum_{d|n} \mu(d) 2^{n/d} & \mbox{if } n \neq 2 \bmod 4, \\ \displaystyle \frac{1}{n}\sum_{d|n} \mu(d) 2^{n/d}+\frac{2}{n}\sum_{d|\frac{n}{2}} \mu(d) 2^{n/2d} & \mbox{if } n = 2\bmod 4. \end{cases}
\end{align*}
\end{theorem}

\begin{proof}
The formula for $|\C_n(213,312)|$ is a consequence of  Theorem~\ref{thm:enumtent} and Corollary~\ref{cor:PPtent}, together with the fact that
a permutation is unimodal if and only if it avoids the patterns $213$ and $312$. The equality with $|\C_n(132,231)|$ follows from Corollary~\ref{cor:-+}.
 The second formula follows from Theorem~\ref{thm:enum2shift}, using that the set of permutations with at most one descent is $[\Av(21)\,\Av(21)]=\Av(321,2143,3142)$ (see~\cite{Atk}).
 Finally, the third formula is a consequence of Corollary \ref{cor: k=2 relation} and Theorem~\ref{thm:enum2shift}, noting that the class of permutations with at most one ascent is $[\Av(12)\,\Av(12)]=\Av(123,2413,3412)$.
\end{proof}

\end{document}